\documentclass[a4paper]{amsart}

\title[Term Structure Models Driven by Poisson Measures]{Term Structure Models Driven by Wiener Process and
Poisson Measures: Existence and Positivity}
\author{Damir Filipovi\'c \and Stefan Tappe \and Josef Teichmann}
\date{April 29, 2009}
\address{Vienna Institute of Finance, University of Vienna, and Vienna University of Economics and Business Administration, Heiligenst\"adter Strasse 46-48, A-1190 Wien, Austria; Vienna University of Technology, Department of Mathematical Methods in Economics, Wiedner Hauptstrasse 8--10, A-1040 Wien, Austria}
\email{stefan.tappe@vif.ac.at, damir.filipovic@vif.ac.at,\newline
jteichma@fam.tuwien.ac.at}
\thanks{The first and second author gratefully acknowledges the support from WWTF (Vienna Science and Technology Fund). The third author gratefully acknowledges the support from the FWF-grant Y 328 (START prize from the Austrian Science Fund).}

\usepackage{amscd}
\usepackage{amsmath}
\usepackage{amssymb}
\usepackage{amsthm}
\usepackage{bbm}
\usepackage{stmaryrd}
\usepackage{eucal}

\newif\ifpdf
\ifx\pdfoutput\undefined
   \pdffalse        % we are not running PDFLaTeX
\else
   \pdfoutput=1     % we are running PDFLaTeX
   \pdftrue
\fi

\ifpdf
   \usepackage[pdftex]{graphicx}
\else
   \usepackage{graphicx}
\fi

\numberwithin{equation}{section} \swapnumbers

\newtheorem{satz}{Satz}[section]

\newtheorem{theorem}[satz]{Theorem}
\newtheorem{proposition}[satz]{Proposition}
\newtheorem{corollary}[satz]{Corollary}
\newtheorem{lemma}[satz]{Lemma}
\newtheorem{assumption}[satz]{Assumption}

\newtheorem{definition}[satz]{Definition}

\newtheorem{remark}[satz]{Remark}

\begin{document}

\maketitle\thispagestyle{empty}

\begin{abstract}
In the spirit of \cite{BKR}, we investigate term structure models
driven by Wiener process and Poisson measures with forward curve
dependent volatilities. This includes a full existence and
uniqueness proof for the corresponding Heath--Jarrow--Morton type
term structure equation. Furthermore, we characterize positivity
preserving models by means of the characteristic coefficients, which was open for jump-diffusions.
Additionally we treat existence, uniqueness and positivity of the Brody-Hughston
equation \cite{brohug:01a,brohug:01b} of interest rate theory with
jumps, an equation which we believe to be very useful for applications. A key role in our investigation is played
by the method of the moving frame, which allows to transform the Heath--Jarrow--Morton--Musiela equation to a time-dependent SDE.

\bigskip

\end{abstract}


\section{Introduction}

Interest rate theory is dealing with zero-coupon bonds, which are subject to a stochastic evolution due to daily trading of related products like coupon bearing bonds, swaps, caps, floors, swaptions, etc. Zero-coupon bonds, which is a financial asset paying the holder one unit of cash at maturity time $T$, are a conceptually important product, since one can easily write all other products as derivatives on them. We do always assume default-free bonds, i.e. there are no counterparty risks in the considered markets. The Heath-Jarrow-Morton methodology takes the bond market as a whole as today's aggregation of information on interest rates and one tries to model future flows of information by a stochastic evolution equation on the set of possible scenarios of bond prices. For the set of possible scenarios of bond prices the forward rate proved to be a flexible and useful parameterization, since it maps possible scenarios of the bond market to open subsets of (Hilbert) spaces of forward rate curves. Under some regularity assumptions the price of a zero coupon bond at $t\le T$ can be
written as
\[ P(t,T)=\exp\bigg(-\int_t^T f(t,u) du\bigg) \]
where $f(t,T)$ is the forward rate for date $T$. We do usually assume the forward rate to be continuous in maturity time $ T $. The classical continuous framework for the evolution of the forward rates goes
back to Heath, Jarrow and Morton (HJM) \cite{HJM}. They assume that,
for every date $T$, the forward rates $f(t,T)$ follow an It\^o
process of the form
\begin{align}\label{hjm-forward-rates-wiener}
df(t,T) = \alpha(t,T) dt + \sum_{j=1}^d \sigma^j(t,T) dW_t^j,\quad
t\in [0,T]
\end{align}
where $W=(W^1,\ldots,W^d)$ is a standard Brownian motion in
$\mathbb{R}^d$.

Empirical studies have revealed that models based on Brownian motion
only provide a poor fit to observed market data. We refer to
\cite[Chap. 5]{Raible}, where it is argued that empirically observed
log returns of zero coupon bonds are not normally distributed, a
fact, which has long before been known for the distributions of
stock returns. Bj\"ork et al.\ \cite{BKR,BKR0}, Eberlein et al.\
\cite{Eberlein-Raible,Eberlein_O, Eberlein_J, Eberlein_K1,
Eberlein_K2,Eberlein_Kluge_Review}  and others (\cite{Shirakawa,
Jarrow_Madan, Hyll}) thus proposed to replace the classical Brownian
motion $W$ in \eqref{hjm-forward-rates-wiener} by a more general
driving noise, also taking into account the occurrence of jumps.
Carmona and Tehranchi \cite{carteh} proposed models based on
infinite dimensional Wiener processes, see also \cite{fillnm}. In
the spirit of Bj\"ork et al.\ \cite{BKR} and Carmona and Tehranchi
\cite{carteh}, we focus on term structure models of the type
\begin{align}\label{hjm-forward-rates}
df(t,T) = \alpha(t,T)dt + \sum_j \sigma^j(t,T) d\beta_t^j + \int_E
\gamma(t,x,T) (\mu(dt,dx) - F(dx)dt),
\end{align}
where $\{ \beta_j \}$ denotes a (possibly infinite) sequence of
real-valued, independent Brownian motions and, in addition, $\mu$ is
a homogeneous Poisson random measure on $\mathbb{R}_+ \times E$ with
compensator $dt \otimes F(dx)$, where $E$ denotes the mark space.

For what follows, it will be convenient to switch to the alternative
parameterization
\begin{align*}
r_t(\xi) := f(t,t+\xi), \quad \xi \geq 0
\end{align*}
which is due to Musiela \cite{Musiela}. Then, we may regard
$(r_t)_{t \geq 0}$ as \textit{one} stochastic process with values in
$H$, that is
\begin{align*}
r : \Omega \times \mathbb{R}_+ \rightarrow H,
\end{align*}
where $H$ denotes a Hilbert space of forward curves $h :
\mathbb{R}_+ \rightarrow \mathbb{R}$ to be specified later. Recall
that we always assume that forward rate curves are continuous.
Denoting by $(S_t)_{t \geq 0}$ the shift semigroup on $H$, that is
$S_t h = h(t + \cdot)$, equation (\ref{hjm-forward-rates}) becomes
in integrated form
\begin{equation}\label{mild-solution-intro-1}
\begin{aligned}
r_t(\xi) &= S_t h_0(\xi) + \int_0^t S_{t-s} \alpha(s,s+\xi)ds +
\sum_j \int_0^t S_{t-s} \sigma^j(s,s+\xi)d\beta_s^j
\\ &\quad + \int_0^t \int_E S_{t-s} \gamma(s,x,s+\xi)
(\mu(ds,dx) - F(dx)ds), \quad t \geq 0
\end{aligned}
\end{equation}
where $h_0 \in H$ denotes the initial forward curve and $S_{t-s}$
operates on the functions $\xi \mapsto \alpha(s,s+\xi)$, $\xi
\mapsto \sigma^j(s,s+\xi)$ and $\xi \mapsto \gamma(s,x,s+\xi)$.

From a financial modeling point of view, one would rather consider
drift and volatilities to be functions of the prevailing forward
curve, that is
\begin{align*}
\alpha &: H \rightarrow H,
\\ \sigma^j &: H \rightarrow H, \quad \text{for all $j$}
\\ \gamma &: H \times E \rightarrow H.
\end{align*}
For example, the volatilities could be of the form $\sigma^j(h) =
\phi_j(\ell_1(h),\ldots,\ell_p(h))$ for some $p \in \mathbb{N}$ with
$\phi_j : \mathbb{R}^p \rightarrow H$ and $\ell_i : H \rightarrow
\mathbb{R}$. We may think of $\ell_i(h) = \frac{1}{\xi_i}
\int_0^{\xi_i} h(\eta) d\eta$ (benchmark yields) or $\ell_i(h) =
h(\xi_i)$ (benchmark forward rates).

The implied bond market
\begin{align}\label{bond-market}
P(t,T) = \exp \bigg( - \int_0^{T-t} r_t(\xi)d\xi \bigg)
\end{align}
is free of arbitrage if we can find an equivalent (local) martingale
measure such that discounted bond prices
$$
\exp \bigg( -\int_0^t r_s(0) ds \bigg) P(t,T)
$$
are local martingales for $ 0 \leq t \leq T $. If we formulate the HJM equation with respect to such an equivalent martingale measure, then the drift is determined by the volatility and jump structure, i.e. $\alpha = \alpha_{\rm HJM} : H
\rightarrow H$ is given by
\begin{align}\label{def-alpha-HJM-intro}
\alpha_{\rm HJM}(h) := \sum_j \sigma^j(h) \Sigma^j(h) - \int_E
\gamma(h,x) \left( e^{\Gamma(h,x)} - 1 \right) F(dx)
\end{align}
for all $h \in H$, where we have set
\begin{align}\label{def-Sigma}
\Sigma^j(h)(\xi) &:= \int_0^{\xi} \sigma^j(h)(\eta)d\eta, \quad
\text{for all $j$}
\\ \label{def-Gamma} \Gamma(h,x)(\xi) &:= -\int_0^{\xi} \gamma(h,x)(\eta) d\eta.
\end{align}
According to \cite{BKR} (if the Brownian motion is infinite
dimensional, see also \cite{fillnm}), condition
(\ref{def-alpha-HJM-intro}) guarantees that the discounted zero
coupon bond prices are local martingales for all maturities $T$, whence the market is
free of arbitrage. In the classical situation, where the model is
driven by a finite dimensional standard Brownian motion,
(\ref{def-alpha-HJM-intro}) is the well-known HJM drift condition
derived in \cite{HJM}.

Our requirements lead to the forward rates $(r_t)_{t \geq 0}$ in
(\ref{mild-solution-intro-1}) being a solution of the stochastic differential
equation
\begin{equation}\label{mild-solution-intro-2}
\begin{aligned}
r_t &= S_t h_0 + \int_0^t S_{t-s} \alpha_{\rm HJM}(r_s)ds + \sum_j
\int_0^t S_{t-s} \sigma^j(r_s)d\beta_s^j
\\ &\quad + \int_0^t \int_E S_{t-s} \gamma(r_{s-},x)
(\mu(ds,dx) - F(dx)ds), \quad t \geq 0
\end{aligned}
\end{equation}
and it arises the question whether this equation possesses a
solution. To our knowledge, there has not yet been an explicit proof
for the \textit{existence} of a solution to the Poisson measure
driven equation (\ref{mild-solution-intro-2}). We thus provide such
a proof in our paper, see Theorem \ref{thm-ex-hjm}. For term
structure models driven by a Brownian motion, the existence proof
has been provided in \cite{fillnm} and for the L\'evy case in
\cite{Filipovic-Tappe}. We also refer to the related papers
\cite{P-Z-paper} and \cite{Marinelli}.

In the spirit of \cite{Da_Prato} and \cite{P-Z-book}, an $H$-valued
stochastic process $(r_t)_{t \geq 0}$ satisfying
(\ref{mild-solution-intro-2}) is a so-called \textit{mild solution}
for the (semi-linear) stochastic partial differential equation
\begin{align}\label{HJMM-eqn}
\left\{
\begin{array}{rcl}
dr_t & = & (\frac{d}{d\xi} r_t + \alpha_{\rm HJM}(r_t))dt + \sum_j
\sigma^j(r_t) d\beta_t^j
\\ & &+\int_E \gamma(r_{t-},x) (\mu(dt,dx) -
F(dx)dt)
\medskip
\\ r_0 & = & h_0,
\end{array}
\right.
\end{align}
where $\frac{d}{d\xi}$ becomes the infinitesimal generator of the
strongly continuous semigroup of shifts $(S_t)_{t \geq 0}$.

In the sequel, we are therefore concerned with establishing the
existence of mild solutions for the HJMM
(Heath--Jarrow--Morton--Musiela) equation (\ref{HJMM-eqn}). As in
\cite{SPDE}, we understand stochastic partial differential equations
as time-dependent transformations of time-dependent stochastic differential
equations with infinite dimensional state space. More precisely, on an enlarged space $\mathcal{H}$ of
forward curves $h : \mathbb{R} \rightarrow \mathbb{R}$, which are
indexed by the whole real line, equipped with the strongly
continuous \textit{group} $(U_t)_{t \in \mathbb{R}}$ of shifts, we
solve the stochastic differential equation
\begin{align}\label{HJM-eqn}
\left\{
\begin{array}{rcl}
df_t & = & U_{-t} \ell \alpha_{\rm HJM}(\pi U_t f_t)dt + \sum_j
U_{-t} \ell \sigma^j(\pi U_t f_t) d\beta_t^j
\\ && + \int_E U_{-t}
\ell \gamma(\pi U_t f_{t-},x) (\mu(dt,dx) - F(dx)dt)
\medskip
\\ f_0 & = & \ell h_0,
\end{array}
\right.
\end{align}
where $\ell : H \rightarrow \mathcal{H}$ is an isometric embedding
and $\pi : \mathcal{H} \rightarrow H$ the adjoint operator of
$\ell$, and afterwards, we transform the solution process $(f_t)_{t
\geq 0}$ by $r_t := \pi U_t f_t$ in order to obtain a mild solution
for (\ref{HJMM-eqn}). Notice that (\ref{HJM-eqn}) just corresponds
to the original HJM dynamics in (\ref{hjm-forward-rates}), where, of
course, the forward rate $f_t(T)$ has no economic interpretation for
$T < t$. Thus, we will henceforth refer to (\ref{HJM-eqn}) as the
HJM (Heath--Jarrow--Morton) equation.

In practice, we are interested in term structure models producing
positive forward curves since negative forward rates are very rarely observed. 
After establishing the existence issue, we
shall therefore focus on positivity preserving term structure
models, and give a characterization of such models. The HJM equation
(\ref{HJM-eqn}) on the enlarged function space will be the key for
analyzing positivity of forward curves. Indeed, the method of the
moving frame, see \cite{SPDE}, allows us to use standard stochastic
analysis (see \cite{Jacod-Shiryaev}) for our investigations. It will
turn out that the conditions
\begin{align*}
&\sigma^j(h)(\xi) = 0, \, \text{for all $\xi \in (0,\infty), \, h
\in \partial P_{\xi}$ and all $j$}
\\ &h + \gamma(h,x) \in P, \, \text{for all $h \in P$ and $F$-almost all $x \in E$}
\\ &\gamma(h,x)(\xi) = 0, \, \text{for all $\xi \in (0,\infty), \, h \in
\partial P_{\xi}$ and $F$-almost all $x \in E$}
\end{align*}
where $P$ denotes the convex cone of all nonnegative forward curves,
and $\partial P_{\xi}$ the set of all nonnegative forward curves $h$
with $h(\xi) = 0$, are necessary and sufficient for the positivity
preserving property, see Theorem \ref{thm-ex-pos}. For this purpose,
we provide a general positivity preserving result, see Theorem
\ref{thm-pos}, which is of independent interest and can also be
applied on other function spaces. Positivity results for the diffusion case have been worked out in
\cite{Kotelenez} and \cite{Milian-inf}. We would like to mention in particular the important and beautiful
work \cite{Nakayama}, where through an application of a general support theorem positivity is proved. This general argument we shall also apply for our reasonings.

Another approach to interest rate markets was suggested by D.~Brody
and L.~Hughston (see \cite{brohug:01a} and \cite{brohug:01b})
inspired by methods from information geometry. It does need the
additional assumption that bond prices behave with respect to
maturity time $T$ like inverse distribution functions. However, this
is economically a completely innocent and even very natural
assumption, since it simply means that the (nominal) short rate
cannot turn negative. In this case possible scenarios of bond prices
are mapped to the set of probability densities on $ \mathbb{R}_+ $.
At first sight this approach seems more delicate, since the set of
probability densities is not a vector space any more but rather a
convex set (with a couple of different topologies). More precisely,
the underlying basic observation is the following: assume a positive
short rate, which almost surely does not converge to $ 0 $, then
bond prices satisfy that
\begin{itemize}
 \item prices as a function of maturity, $ T \mapsto P(t,T) $, are decreasing and
 continuous,
 \item the limit for $ T \to \infty $ is $ 0 $.
\end{itemize}
This suggests the following parameterization of bond price  -- under slight additional regularity assumptions --
$$
P(t,T) = \int_{T-t}^{\infty} \rho(t,u) du,
$$
where $ t \mapsto \rho(t,.) $ is a stochastic process of probability
densities on $ \mathbb{R}_+ $. Assuming additionally that $ t
\mapsto \rho_t(0) $ is a well-defined stochastic process, the short
rate, we can consider no arbitrage conditions. If one assumes the
existence of an equivalent martingale measure for the discounted
bond prices and its strict positivity, then -- with respect to this equivalent martingale
measure -- the process $ \rho_t $ satisfies the following equation
\begin{align}\label{bh-equation_intro}
\frac{d \rho_t(\xi)}{\rho_t(\xi)} = \bigg( \rho_t (0) +
\frac{d}{d\xi} (\log\rho_t(\xi)) \bigg) dt + \sum_{j} \left(
\sigma^j(t,\xi) - \overline{\sigma^j(t,\cdot)} \right) d\beta^j_t.
\end{align}
Here $ \overline{\sigma(t,\cdot)} $ denotes the average of $
\sigma(t,.) $ with respect to the probability measure $
\rho(\xi)d\xi $, see Section \ref{sec_bh-equation} for details. We
first extend this setting, which was basically outlined in
\cite{brohug:01a} and \cite{brohug:01b}, into the realm of jump
processes. Second, we prove existence and uniqueness of the
resulting equations on appropriate Hilbert spaces of probability
densities with or without jumps, a problem which has been left open
in the literature so far. It has to be pointed out that we do not pursue the approach suggested in
\cite{brohug:01a} to map densities via to the unit sphere of an appropriate $ L^2 $, but that we treat
equation \eqref{bh-equation_intro} directly by embedding probability densities into an appropriate Hilbert space of functions.

The remainder of this text is organized as follows. In Section
\ref{sec-space} we introduce the space $H_{\beta}$ of forward
curves. Using this space, we prove in Section \ref{sec-ex-tsm},
under appropriate regularity assumptions, the existence of a unique
solution for the HJMM equation (\ref{HJMM-eqn}). The positivity
issue of term structure models is treated in Section
\ref{sec-positivity}, there we show first the necessary conditions with a general semimartingale argument. The sufficient conditions are proved to hold true by switching on the jumps ``slowly''. This allows for a reduction to results from \cite{Nakayama}. An alternative approach to HJMM is proposed in
Section \ref{sec_bh-equation}, where the established (general)
positivity results are applied to the Brody-Hughston equation from
interest rate theory (with jumps). For convenience of the reader, we
provide the prerequisites on stochastic partial differential
equations in Appendix \ref{app-SDE}.

\section{The space of forward curves}\label{sec-space}

In this section, we introduce the space of forward curves, on which
we will solve the HJMM equation (\ref{HJMM-eqn}) in Section
\ref{sec-ex-tsm}.

We fix an arbitrary constant $\beta > 0$. Let $H_{\beta}$ be the
space of all absolutely continuous functions $h : \mathbb{R}_+
\rightarrow \mathbb{R}$ such that
\begin{align*}
\| h \|_{\beta} := \bigg( |h(0)|^2 + \int_{\mathbb{R}_+} |h'(\xi)|^2
e^{\beta \xi} d\xi \bigg)^{\frac{1}{2}} < \infty.
\end{align*}
Let $(S_t)_{t \geq 0}$ be the shift semigroup on $H_{\beta}$ defined
by $S_t h := h(t + \cdot)$ for $t \in \mathbb{R}_+$.

Since forward curves should flatten for large time to maturity $\xi$, the
choice of $H_{\beta}$ is reasonable from an economic point of view.

Moreover, let $\mathcal{H}_{\beta}$ be the space of all absolutely
continuous functions $h : \mathbb{R} \rightarrow \mathbb{R}$ such
that
\begin{align*}
\| h \|_{\beta} := \bigg( |h(0)|^2 + \int_{\mathbb{R}} |h'(\xi)|^2
e^{\beta|\xi|} d\xi \bigg)^{\frac{1}{2}} < \infty.
\end{align*}
Let $(U_t)_{t \in \mathbb{R}}$ be the shift group on
$\mathcal{H}_{\beta}$ defined by $U_t h := h(t + \cdot)$ for $t \in
\mathbb{R}$.

The linear operator $\ell : H_{\beta} \rightarrow
\mathcal{H}_{\beta}$ defined by
\begin{align*}
\ell(h)(\xi) :=
\begin{cases}
h(0), & \xi < 0
\\ h(\xi), & \xi \geq 0,
\end{cases}
\quad h \in H_{\beta}
\end{align*}
is an isometric embedding with adjoint operator $\pi := \ell^* :
\mathcal{H}_{\beta} \rightarrow H_{\beta}$ given by $\pi(h) =
h|_{\mathbb{R}_+}$, $h \in \mathcal{H}_{\beta}$.

\begin{theorem}\label{thm-group}
Let $\beta > 0$ be arbitrary.
\begin{enumerate}
\item The space $(H_{\beta},\| \cdot \|_{\beta})$ is a separable
Hilbert space.

\item For each $\xi \in \mathbb{R}_+$, the point
evaluation $h \mapsto h(\xi) : H_{\beta} \rightarrow \mathbb{R}$ is
a continuous linear functional.

\item $(S_t)_{t \geq 0}$ is a $C_0$-semigroup on $H_{\beta}$ with infinitesimal generator $\frac{d}{d\xi} :
\mathcal{D}(\frac{d}{d\xi}) \subset H_{\beta} \rightarrow
H_{\beta}$, $\frac{d}{d\xi}h = h'$, and domain
\begin{align*}
\mathcal{D} ({\textstyle \frac{d}{d\xi}}) = \{ h \in
H_{\beta} \, | \, h' \in H_{\beta} \}.
\end{align*}

\item Each $h \in H_{\beta}$ is continuous, bounded and
the limit $h(\infty) := \lim_{\xi \rightarrow \infty} h(\xi)$
exists.

\item $H_{\beta}^0 := \{ h \in H_{\beta} \, |
\, h(\infty) = 0 \}$ is a closed subspace of $H_{\beta}$.

\item There are universal constants
$C_1,C_2,C_3,C_4
> 0$, only depending on $\beta$, such that for all $h \in H_{\beta}$ we have the estimates
\begin{align}
\label{estimate-c1} \| h' \|_{L^1(\mathbb{R}_+)} &\leq C_1 \| h
\|_{\beta},
\\ \label{estimate-c2} \| h \|_{L^{\infty}(\mathbb{R}_+)} &\leq C_2 \| h \|_{\beta},
\\ \label{est-h} \| h - h(\infty) \|_{L^1(\mathbb{R}_+)} &\leq C_3 \| h \|_{\beta},
\\ \label{est-h-4} \| (h - h(\infty))^4 e^{\beta \bullet} \|_{L^1(\mathbb{R}_+)} &\leq C_4 \| h
\|_{\beta}^4.
\end{align}
\item For each $\beta' > \beta$, we have $H_{\beta'} \subset
H_{\beta}$, the relation
\begin{align}\label{est-embedding}
\| h \|_{\beta} \leq \| h \|_{\beta'}, \quad h \in H_{\beta'}
\end{align}
and there is a universal constant $C_5 > 0$, only depending on
$\beta$ and $\beta'$, such that for all $h \in H_{\beta'}$ we have
the estimate
\begin{align}\label{est-C5}
\| (h - h(\infty))^2 e^{\beta \bullet} \|_{L^1(\mathbb{R}_+)} \leq
C_5 \| h \|_{\beta'}^2.
\end{align}

\item The space $(\mathcal{H}_{\beta},\| \cdot \|_{\beta})$ is a separable
Hilbert space, $(U_t)_{t \in \mathbb{R}}$ is a $C_0$-group on
$\mathcal{H}_{\beta}$ and, for each $\xi \in \mathbb{R}$, the point
evaluation $h \mapsto h(\xi)$, $\mathcal{H}_{\beta} \rightarrow
\mathbb{R}$ is a continuous linear functional.

\item The diagram
\[ \begin{CD}
\mathcal{H}_{\beta} @>U_t>> \mathcal{H}_{\beta}\\
@AA\ell A @VV\pi V\\
H_{\beta} @>S_t>> H_{\beta}
\end{CD} \]
commutes for every $t \in \mathbb{R}_+$, that is
\begin{align}\label{diagram-commutes}
\pi U_t \ell h = S_t h \quad \text{for all $t \in \mathbb{R}_+$ and
$h \in H_{\beta}$.}
\end{align}

\end{enumerate}
\end{theorem}

\begin{proof}
Note that $H_{\beta}$ is the space $H_w$ from \cite[Sec. 5.1]{fillnm}
with weight function $w(\xi) = e^{\beta \xi}$, $\xi \in
\mathbb{R}_+$. Hence, the first six statements follow from
\cite[Thm. 5.1.1, Cor. 5.1.1]{fillnm}.

For each $\beta' > \beta$, the observation
\begin{align*}
\int_{\mathbb{R}_+} |h'(\xi)|^2 e^{\beta \xi} d\xi \leq
\int_{\mathbb{R}_+} |h'(\xi)|^2 e^{\beta' \xi} d\xi, \quad h \in
H_{\beta'}
\end{align*}
shows $H_{\beta'} \subset H_{\beta}$ and (\ref{est-embedding}). For
an arbitrary $h \in H_{\beta'}$ we have, by H\"older's inequality,
\begin{align*}
&\int_{\mathbb{R}_+} |h(\xi) - h(\infty)|^2 e^{\beta \xi} d \xi =
\int_{\mathbb{R}_+} \bigg( \int_{\xi}^{\infty} h'(\eta)
e^{\frac{1}{2} \beta' \eta} e^{-\frac{1}{2} \beta' \eta} d \eta
\bigg)^2 e^{\beta \xi} d\xi
\\ &\leq \int_{\mathbb{R}_+} \bigg( \int_{\mathbb{R}_+} |h'(\eta)|^2 e^{\beta' \eta} d\eta
\bigg) \bigg( \int_{\xi}^{\infty} e^{- \beta' \eta} d\eta \bigg)
e^{\beta \xi} d \xi \leq \frac{1}{\beta'(\beta' - \beta)} \| h
\|_{\beta'}^2.
\end{align*}
Choosing $C_5 := \frac{1}{\beta'(\beta' - \beta)}$ proves
(\ref{est-C5}).

It is clear that $\| \cdot \|_{\beta}$ is a norm on
$\mathcal{H}_{\beta}$. First, we prove that there is a constant $K_1
> 0$ such that
\begin{align}\label{K1-1}
\| h' \|_{L^1(\mathbb{R})} \leq K_1 \| h \|_{\beta}, \quad h \in
\mathcal{H}_{\beta}.
\end{align}
Setting $K_1 := \sqrt{\frac{2}{\beta}}$, this is established by
H\"older's inequality
\begin{equation}\label{K1-derive}
\begin{aligned}
&\int_{\mathbb{R}} |h'(\xi)| d\xi = \int_{\mathbb{R}} |h'(\xi)|
e^{\frac{1}{2} \beta |\xi|} e^{-\frac{1}{2} \beta |\xi|} d\xi
\\ &\leq \bigg( \int_{\mathbb{R}} |h'(\xi)|^2 e^{\beta |\xi|} d \xi \bigg)^{\frac{1}{2}}
\bigg( \int_{\mathbb{R}} e^{-\beta |\xi|} d\xi \bigg)^{\frac{1}{2}}
= \sqrt{\frac{2}{\beta}} \bigg( \int_{\mathbb{R}} |h'(\xi)|^2
e^{\beta |\xi|} d \xi \bigg)^{\frac{1}{2}}.
\end{aligned}
\end{equation}
As a consequence of (\ref{K1-1}), for each $h \in
\mathcal{H}_{\beta}$ the limits $h(\infty) := \lim_{\xi \rightarrow
\infty} h(\xi)$ and $h(-\infty) := \lim_{\xi \rightarrow -\infty}
h(\xi)$ exist. This allows us to the define the new norm
\begin{align*}
| h |_{\beta} := \bigg( |h(-\infty)|^2 + \int_{\mathbb{R}}
|h'(\xi)|^2 e^{\beta|\xi|} d\xi \bigg)^{\frac{1}{2}}, \quad h \in
\mathcal{H}_{\beta}.
\end{align*}

From (\ref{K1-derive}) we also deduce that
\begin{align}\label{K1-2}
\| h' \|_{L^1(\mathbb{R})} \leq K_1 | h |_{\beta}, \quad h \in
\mathcal{H}_{\beta}.
\end{align}
Setting $K_2 := 1 + K_1$, from (\ref{K1-1}) and (\ref{K1-2}) is
follows that
\begin{align}\label{K2-1}
\| h \|_{L^{\infty}(\mathbb{R})} &\leq K_2 \| h \|_{\beta},
\\ \label{K2-2} \| h \|_{L^{\infty}(\mathbb{R})} &\leq K_2 | h |_{\beta}
\end{align}
for all $h \in \mathcal{H}_{\beta}$. Estimate (\ref{K2-1}) shows
that, for each $\xi \in \mathbb{R}$, the point evaluation $h \mapsto
h(\xi)$, $\mathcal{H}_{\beta} \rightarrow \mathbb{R}$ is a
continuous linear functional.

Using (\ref{K2-1}) and (\ref{K2-2}) we conclude that
\begin{align*}
\frac{1}{(1 + K_2^2)^{\frac{1}{2}}} \| h \|_{\beta} \leq |h|_{\beta}
\leq (1 + K_2^2)^{\frac{1}{2}} \| h \|_{\beta}, \quad h \in
\mathcal{H}_{\beta}
\end{align*}
which shows that $\| \cdot \|_{\beta}$ and $| \cdot |_{\beta}$ are
equivalent norms on $\mathcal{H}_{\beta}$.

Consider the separable Hilbert space $\mathbb{R} \times
L^2(\mathbb{R})$ equipped with the norm $(|\cdot|^2 + \| \cdot
\|_{L^2(\mathbb{R})}^2)^{\frac{1}{2}}$. Then the linear operator $T
: (\mathcal{H}_{\beta},|\cdot|_{\beta}) \rightarrow \mathbb{R}
\times L^2(\mathbb{R})$ given by
\begin{align*}
T h := (h(-\infty),h' e^{\frac{1}{2} \beta |\cdot|}), \quad h \in
\mathcal{H}_{\beta}
\end{align*}
is an isometric isomorphism with inverse
\begin{align*}
(T^{-1}(u,g))(x) = u + \int_{-\infty}^x g(\eta) e^{-\frac{1}{2}
\beta |\eta|} d\eta, \quad (u,g) \in \mathbb{R} \times
L^2(\mathbb{R}).
\end{align*}
Since $\| \cdot \|_{\beta}$ and $|\cdot|_{\beta}$ are equivalent,
$(\mathcal{H}_{\beta},\| \cdot \|_{\beta})$ is a separable Hilbert
space.

Next, we claim that
\begin{align*}
\mathcal{D}_0 := \{ g \in \mathcal{H}_{\beta} \,|\, g' \in
\mathcal{H}_{\beta} \}
\end{align*}
is dense in $\mathcal{H}_{\beta}$. Indeed,
$C_c^{\infty}(\mathbb{R})$ is dense in $L^2(\mathbb{R})$, see
\cite[Cor. IV.23]{Brezis}. Fix $h \in \mathcal{H}_{\beta}$ and let
$(g_n)_{n \in \mathbb{N}} \subset C_c^{\infty}(\mathbb{R})$ be an
approximating sequence of $h' e^{\frac{1}{2} \beta |\cdot|}$ in
$L^2(\mathbb{R})$. Then we have $h_n := T^{-1}(h(-\infty),g_n) \in
\mathcal{D}_0$ for all $n \in \mathbb{N}$ and $h_n \rightarrow h$ in
$\mathcal{H}_{\beta}$.

For each $t \in \mathbb{R}$ and $h \in \mathcal{H}_{\beta}$, the
function $U_t h$ is again absolutely continuous. We claim that there
exists a constant $K_3 > 0$ such that
\begin{align}\label{est-strong-cont}
\| U_t h \|_{\beta}^2 \leq (K_3 + e^{\beta |t|}) \| h \|_{\beta}^2,
\quad (t,h) \in \mathbb{R} \times \mathcal{H}_{\beta}.
\end{align}
Using (\ref{K2-1}), we obtain
\begin{align*}
\| U_t h \|_{\beta}^2 &= |h(t)|^2 + \int_0^{\infty} |h'(\xi+t)|^2
e^{\beta \xi} d\xi + \int_{-\infty}^0 |h'(\xi+t)|^2 e^{-\beta \xi}
d\xi
\\ &= |h(t)|^2 + e^{- \beta t} \int_t^{\infty} |h'(\xi)|^2 e^{\beta
\xi} d\xi + e^{\beta t} \int_{-\infty}^t |h'(\xi)|^2 e^{-\beta \xi}
d\xi
\\ &\leq (K_2^2 + 1 + e^{\beta |t|}) \| h \|_{\beta}^2, \quad h \in
\mathcal{H}_{\beta}.
\end{align*}
Setting $K_3 := 1 + K_2^2$, this establishes
(\ref{est-strong-cont}). Hence, we have $U_t h \in
\mathcal{H}_{\beta}$ for all $t \in \mathbb{R}$ and $h \in
\mathcal{H}_{\beta}$ and $U_t \in L(\mathcal{H}_{\beta})$, $t \in
\mathbb{R}$.

It remains to show strong continuity of the group $(U_t)_{t \in
\mathbb{R}}$. Using the observation
\begin{align*}
h(\xi+t) - h(\xi) = t \int_0^1 h'(\xi + st)ds, \quad (\xi,t,h) \in
\mathbb{R} \times \mathbb{R} \times \mathcal{H}_{\beta},
\end{align*}
which holds everywhere for an appropriately chosen absolutely continuous representative of $ h \in \mathcal{H}_{\beta} $, and (\ref{est-strong-cont}), we obtain for each $g \in \mathcal{D}_0$ the convergence
\begin{align*}
&\| U_t g - g \|_{\beta}^2 = |g(t) - g(0)|^2 + \int_{\mathbb{R}}
|g'(\xi+t) - g'(\xi)|^2 e^{\beta |\xi|} d\xi
\\ &\leq |g(t) - g(0)|^2 + t^2 \int_0^1 \int_{\mathbb{R}}
|g''(\xi+st)|^2 e^{\beta |\xi|} d\xi ds
\\ &\leq |g(t) - g(0)|^2 + t^2 \int_0^1 \| U_{st} g' \|_{\beta}^2 ds
\\ &\leq |g(t) - g(0)|^2 + t^2 \| g' \|_{\beta}^2 \int_0^1 (K_3 + e^{\beta s|t|}) ds
\\ &= |g(t) - g(0)|^2 + \bigg( K_3 t^2 + \frac{|t|}{\beta}
(e^{\beta |t|} - 1) \bigg) \| g' \|_{\beta}^2 \rightarrow 0 \quad
\text{as $t \rightarrow 0$.}
\end{align*}
Hence, $(U_t)_{t \in \mathbb{R}}$ is strongly continuous on
$\mathcal{D}_0$. But for any $h \in \mathcal{H}_{\beta}$ and
$\epsilon > 0$ there exists $g \in \mathcal{D}_0$ with $\| h-g
\|_{\beta} < \frac{\epsilon}{4 \sqrt{K_3 + e^{\beta}}}$. Combining
this with (\ref{est-strong-cont}) yields
\begin{align*}
\| U_t h - h \|_{\beta} &\leq \| U_t(h-g) \|_{\beta} + \| U_t g - g
\|_{\beta} + \| g-h \|_{\beta}
\\ &< \sqrt{K_3 + e^{\beta |t|}} \frac{\epsilon}{4 \sqrt{K_3 + e^{\beta}}} + \| U_t g - g \|_{\beta} +
\frac{\epsilon}{4 \sqrt{K_3 + e^{\beta}}} < \epsilon
\end{align*}
for $t \in \mathbb{R}$ small enough. We conclude that $(U_t)_{t \in
\mathbb{R}}$ is a $C_0$-group on $\mathcal{H}_{\beta}$.

Finally, relation (\ref{diagram-commutes}) follows from the
definitions of $\ell$ and $\pi$.
\end{proof}

\section{Existence of term structure models driven by Wiener process and Poisson
measures}\label{sec-ex-tsm}

In this section, we establish existence and uniqueness of the HJMM
equation (\ref{HJMM-eqn}) with diffusive and jump components on the Hilbert spaces introduced in the last
section.

Let $0 < \beta < \beta'$ be arbitrary real numbers. The framework is
the same as in Appendix \ref{app-SDE} with $H = H_{\beta}$ being the
space of forward curves introduced in Section \ref{sec-space},
equipped with the strongly continuous semigroup $(S_t)_{t \geq 0}$
of shifts, which has the infinitesimal generator $A = \frac{d}{d
\xi}$.

Let $\sigma : H_{\beta} \rightarrow L_2^0(H_{\beta}^0)$ and $\gamma
: H_{\beta} \times E \rightarrow H_{\beta'}^0$. For each $j$ we
define $\sigma^j : H_{\beta} \rightarrow H_{\beta}^0$ as
$\sigma^j(h) := \sqrt{\lambda_j} \sigma(h) e_j$.

\begin{assumption}\label{ass-Lipschitz}
We assume there exists a measurable function $\Phi : E \rightarrow
\mathbb{R}_+$ satisfying
\begin{align}\label{est-Gamma-f}
&|\Gamma(h,x)(\xi)| \leq \Phi(x), \quad h \in H_{\beta}, \, x \in E
\text{ and } \xi \in \mathbb{R}_+
\end{align}
a constant $L > 0$ such that
\begin{align}\label{Lipschitz-sigma-HJM}
\| \sigma(h_1) - \sigma(h_2) \|_{L_2^0(H_{\beta})} &\leq L \| h_1 -
h_2 \|_{\beta}
\\ \label{Lipschitz-gamma-HJM} \bigg( \int_E e^{\Phi(x)} \| \gamma(h_1,x) - \gamma(h_2,x) \|_{\beta'}^2 F(dx)
\bigg)^{\frac{1}{2}} &\leq L \| h_1 - h_2 \|_{\beta}
\end{align}
for all $h_1,h_2 \in H_{\beta}$, and a constant $M > 0$ such that
\begin{align}\label{bounded-sigma-HJM}
\| \sigma(h) \|_{L_2^0(H_{\beta})} &\leq M
\\ \label{bounded-gamma-HJM} \int_E e^{\Phi(x)} ( \| \gamma(h,x) \|_{\beta'}^2 \vee \| \gamma(h,x) \|_{\beta'}^4 )
F(dx) &\leq M
\end{align}
for all $h \in H_{\beta}$. Furthermore, we assume that for each $h
\in H_{\beta}$ the map
\begin{align}\label{def-alpha-2}
\alpha_2(h) &:= - \int_E \gamma(h,x) \left( e^{\Gamma(h,x)} - 1
\right) F(dx),
\end{align}
is absolutely continuous with weak derivative
\begin{equation}\label{der-alpha-2}
\begin{aligned}
\frac{d}{d \xi} \alpha_2(h) = \int_{E} \gamma(h,x)^2 e^{\Gamma(h,x)}
F(dx) - \int_{E} \frac{d}{d \xi} \gamma(h,x) \left( e^{\Gamma(h,x)}
- 1 \right) F(dx).
\end{aligned}
\end{equation}
\end{assumption}

\begin{proposition}
Suppose Assumption \ref{ass-Lipschitz} is fulfilled. Then we have
$\alpha_{\rm HJM}(H_{\beta}) \subset H_{\beta}^0$ and there is a
constant $K > 0$ such that
\begin{align}\label{Lipschitz-desired}
\| \alpha_{\rm HJM}(h_1) - \alpha_{\rm HJM}(h_2) \|_{\beta} \leq K
\| h_1 - h_2 \|_{\beta}
\end{align}
for all $h_1,h_2 \in H_{\beta}$.
\end{proposition}

\begin{proof}
Note that $\alpha_{\rm HJM} = \alpha_1 + \alpha_2$, where
\begin{align*}
\alpha_1(h) := \sum_{j} \sigma^j(h) \Sigma^j(h), \quad h \in
H_{\beta}
\end{align*}
and $\alpha_2$ is given by (\ref{def-alpha-2}). By \cite[Cor.
5.1.2]{fillnm} we have $\sigma^j(h) \Sigma^j(h) \in H_{\beta}^0$, $h
\in H_{\beta}$ for all $j$. Let $h \in H_{\beta}$ be arbitrary. For
$n,m \in \mathbb{N}$ with $n < m$ we have, using \cite[Cor.
5.1.2]{fillnm} again,
\begin{align*}
\bigg\| \sum_{j=n+1}^m \sigma^j(h) \Sigma^j(h) \bigg\|_{\beta} \leq
\sum_{j=n+1}^m \| \sigma^j(h) \Sigma^j(h) \|_{\beta} \leq
\sqrt{3(C_3^2 + 2 C_4)} \sum_{j=n+1}^m \| \sigma^j(h) \|_{\beta}^2.
\end{align*}
Hence, $\sum_j \sigma^j(h) \Sigma^j(h)$ is a Cauchy sequence in
$H_{\beta}^0$, because $\sigma(h) \in L_2^0(H_{\beta}^0)$. We deduce
that $\alpha_1(H_\beta) \subset H_{\beta}^0$.

For all $h \in H_{\beta}$, $x \in E$ and $\xi \in \mathbb{R}_+$ we
have by (\ref{estimate-c2}) and (\ref{est-embedding})
\begin{align}\label{est-gamma}
|\gamma(h,x)(\xi)| \leq C_2 \| \gamma(h,x) \|_{\beta} \leq C_2 \|
\gamma(h,x) \|_{\beta'}.
\end{align}
For all $x \in E$ and $\xi \in \mathbb{R}_+$ we have by
(\ref{est-Gamma-f}), (\ref{est-h}) and (\ref{est-embedding})
\begin{align}\label{est-Gamma}
|e^{\Gamma(h,x)(\xi)} - 1| \leq e^{\Phi(x)} |\Gamma(h,x)(\xi)| \leq
e^{\Phi(x)} \| \gamma(h,x) \|_{L^1(\mathbb{R}_+)} \leq C_3
e^{\Phi(x)} \| \gamma(h,x) \|_{\beta'}.
\end{align}
Estimates (\ref{est-gamma}), (\ref{est-Gamma}) and
(\ref{bounded-gamma-HJM}) show that $\lim_{\xi \rightarrow \infty}
\alpha_2(h)(\xi) = 0$. From (\ref{est-Gamma-f}), (\ref{est-gamma}),
(\ref{bounded-gamma-HJM}) and (\ref{est-C5}) it follows that
\begin{align*}
&\int_{\mathbb{R}_+} \left( \int_E \gamma(h,x)(\xi)^2
e^{\Gamma(h,x)(\xi)} F(dx) \right)^2 e^{\beta \xi} d\xi
\\ &\leq C_2^2
M \int_{\mathbb{R}_+} \left( \int_E \gamma(h,x)(\xi)^2
e^{\Gamma(h,x)(\xi)} F(dx) \right) e^{\beta \xi} d\xi
\\ &\leq C_2^2 M \int_E e^{\Phi(x)} \int_{\mathbb{R}_+}
\gamma(h,x)(\xi)^2 e^{\beta \xi} d\xi F(dx)
\\ &\leq C_2^2 M C_5 \int_E e^{\Phi(x)} \| \gamma(h,x) \|_{\beta'}^2
F(dx) \leq C_2^2 M^2 C_5.
\end{align*}
We obtain by (\ref{est-Gamma}), H\"older's inequality,
(\ref{bounded-gamma-HJM}) and (\ref{est-embedding})
\begin{align*}
&\int_{\mathbb{R}_+} \left( \int_E \frac{d}{d \xi} \gamma(h,x)(\xi)
\left( e^{\Gamma(h,x)(\xi)} - 1 \right) F(dx) \right)^2 e^{\beta
\xi} d\xi
\\ &\leq C_3^2 \int_{\mathbb{R}_+} \bigg( \int_E \bigg| \frac{d}{d\xi} \gamma(h,x)(\xi) \bigg|
e^{\frac{1}{2}\Phi(x)} e^{\frac{1}{2}\Phi(x)} \| \gamma(h,x)
\|_{\beta'} F(dx) \bigg)^2 e^{\beta \xi} d\xi
\\ &\leq C_3^2 M \int_E e^{\Phi(x)} \int_{\mathbb{R}_+} \bigg| \frac{d}{d\xi}
\gamma(h,x)(\xi) \bigg|^2 e^{\beta \xi} d\xi F(dx)
\\ &\leq C_3^2 M \int_E e^{\Phi(x)} \| \gamma(h,x) \|_{\beta'}^2 F(dx)
\leq C_3^2 M^2.
\end{align*}
We conclude that $\alpha_2(H_{\beta}) \subset H_{\beta}^0$, and
hence $\alpha_{\rm HJM}(H_{\beta}) \subset H_{\beta}^0$.

Let $h_1,h_2 \in H_{\beta}$ be arbitrary. By \cite[Cor.
5.1.2]{fillnm}, H\"older's inequality, (\ref{Lipschitz-sigma-HJM})
and (\ref{bounded-sigma-HJM}) we have
\begin{align*}
&\| \alpha_1(h_1) - \alpha_1(h_2) \|_{\beta} \\
& \leq \sqrt{3(C_3^2 + 2
C_4)} \sum_{j} (\| \sigma^j(h_1) \|_{\beta} + \| \sigma^j(h_2)
\|_{\beta}) \| \sigma^j(h_1) - \sigma^j(h_2) \|_{\beta}
\\ &\leq \sqrt{3(C_3^2 + 2 C_4)} \sqrt{\sum_j (\| \sigma^j(h_1) \|_{\beta} + \| \sigma^j(h_2) \|_{\beta})^2}
\sqrt{\sum_j \| \sigma^j(h_1) - \sigma^j(h_2) \|_{\beta}^2}
\\ &\leq \sqrt{6(C_3^2 + 2 C_4)} (\| \sigma(h_1) \|_{L_2^0(H_{\beta})} + \| \sigma(h_2)
\|_{L_2^0(H_{\beta})}) \| \sigma(h_1) - \sigma(h_2)
\|_{L_2^0(H_{\beta})}
\\ &\leq 2ML \sqrt{6(C_3^2 + 2 C_4)} \| h_1 - h_2 \|_{\beta}.
\end{align*}
Furthermore, by (\ref{der-alpha-2}),
\begin{align*}
\| \alpha_2(h_1) - \alpha_2(h_2) \|_{\beta}^2 \leq 4(I_1 + I_2 + I_3
+ I_4),
\end{align*}
where we have put
\begin{align*}
I_1 &:= \int_{\mathbb{R}_+} \left( \int_E \gamma(h_1,x)(\xi)^2
\left( e^{\Gamma(h_1,x)(\xi)} - e^{\Gamma(h_2,x)(\xi)} \right) F(dx)
\right)^2 e^{\beta \xi} d\xi,
\\ I_2 &:= \int_{\mathbb{R}_+} \left( \int_E e^{\Gamma(h_2,x)(\xi)}
(\gamma(h_1,x)(\xi)^2 - \gamma(h_2,x)(\xi)^2) F(dx) \right)^2
e^{\beta \xi} d\xi,
\\ I_3 &:= \int_{\mathbb{R}_+} \left( \int_E
\frac{d}{d \xi} \gamma(h_1,x)(\xi) \left( e^{\Gamma(h_1,x)(\xi)} -
e^{\Gamma(h_2,x)(\xi)} \right) F(dx) \right)^2 e^{\beta \xi} d \xi,
\\ I_4 &:= \int_{\mathbb{R}_+} \left( \int_E \left(
e^{\Gamma(h_2,x)(\xi)} - 1 \right) \left( \frac{d}{d \xi}
\gamma(h_1,x)(\xi) - \frac{d}{d \xi} \gamma(h_2,x)(\xi) \right)
F(dx) \right)^2 e^{\beta \xi} d\xi.
\end{align*}
We get for all $x \in E$ and $\xi \in \mathbb{R}_+$ by
(\ref{est-Gamma-f}), (\ref{est-h}) and (\ref{est-embedding})
\begin{equation}\label{Gamma-exponential}
\begin{aligned}
&|e^{\Gamma(h_1,x)(\xi)} - e^{\Gamma(h_2,x)(\xi)}| \leq e^{\Phi(x)}
|\Gamma(h_1,x)(\xi) - \Gamma(h_2,x)(\xi)|
\\ &\leq e^{\Phi(x)} \| \gamma(h_1,x) - \gamma(h_2,x)
\|_{L^1(\mathbb{R}_+)} \leq C_3 e^{\Phi(x)} \| \gamma(h_1,x) -
\gamma(h_2,x) \|_{\beta'}.
\end{aligned}
\end{equation}
Relations (\ref{Gamma-exponential}), H\"older's inequality,
(\ref{Lipschitz-gamma-HJM}), (\ref{est-h-4}), (\ref{est-embedding})
and (\ref{bounded-gamma-HJM}) give us
\begin{align*}
I_1 &\leq C_3^2 \int_{\mathbb{R}_+} \bigg( \int_E \gamma(h,x)(\xi)^2
e^{\frac{1}{2} \Phi(x)} e^{\frac{1}{2} \Phi(x)} \| \gamma(h_1,x) -
\gamma(h_2,x) \|_{\beta'} F(dx) \bigg)^2 e^{\beta \xi} d\xi
\\ &\leq C_3^2 L^2 \| h_1 - h_2 \|_{\beta}^2 \int_E e^{\Phi(x)} \int_{\mathbb{R}_+}
\gamma(h_1,x)(\xi)^4 e^{\beta \xi} d\xi F(dx)
\\ &\leq C_3^2 L^2 C_4 \| h_1 - h_2 \|_{\beta}^2 \int_E e^{\Phi(x)} \| \gamma(h_1,x) \|_{\beta'}^4 F(dx)
\leq C_3^2 L^2 C_4 M \| h_1 - h_2 \|_{\beta}^2.
\end{align*}
For every $\xi \in \mathbb{R}_+$ we obtain by (\ref{est-gamma}) and
(\ref{bounded-gamma-HJM})
\begin{equation}\label{I3-pre}
\begin{aligned}
&\int_E e^{\Phi(x)} (\gamma(h_1,x)(\xi) + \gamma(h_2,x)(\xi))^2
F(dx)
\\ &\leq 2 \int_E e^{\Phi(x)} (\gamma(h_1,x)(\xi)^2 + \gamma(h_2,x)(\xi)^2) F(dx)
\\ &\leq 2 C_2^2 \bigg( \int_E e^{\Phi(x)} \| \gamma(h_1,x)
\|_{\beta'}^2 F(dx) + \int_E e^{\Phi(x)} \| \gamma(h_2,x)
\|_{\beta'}^2 F(dx) \bigg) \leq 4 C_2^2 M.
\end{aligned}
\end{equation}
Using (\ref{est-Gamma-f}), H\"older's inequality, (\ref{I3-pre}),
(\ref{est-C5}) and (\ref{Lipschitz-gamma-HJM}) we get
\begin{align*}
I_2 &\leq \int_{\mathbb{R}_+} \bigg( \int_E (\gamma(h_1,x)(\xi) +
\gamma(h_2,x)(\xi)) e^{\frac{1}{2}\Phi(x)} e^{\frac{1}{2}\Phi(x)}
(\gamma(h_1,x)(\xi) - \gamma(h_2,x)(\xi)) \bigg)^2 \\
& \qquad \qquad \times e^{\beta \xi}
d\xi
\\ &\leq 4 C_2^2 M \int_E e^{\Phi(x)} \int_{\mathbb{R}_+} (\gamma(h_1,x)(\xi) - \gamma(h_2,x)(\xi))^2 e^{\beta \xi} d\xi F(dx)
\\ &\leq 4 C_2^2 M C_5 \int_E e^{\Phi(x)} \| \gamma(h_1,x)(\xi) - \gamma(h_2,x)(\xi) \|_{\beta'}^2
F(dx) \\
& \leq 4 C_2^2 M C_5 L^2 \| h_1 - h_2 \|_{\beta}^2.
\end{align*}
Using (\ref{Gamma-exponential}), H\"older's inequality,
(\ref{Lipschitz-gamma-HJM}), (\ref{est-embedding}) and
(\ref{bounded-gamma-HJM}) gives us
\begin{align*}
I_3 &\leq C_3^2 \int_{\mathbb{R}_+} \bigg( \int_E \bigg| \frac{d}{d
\xi} \gamma(h_1,x)(\xi) \bigg| e^{\frac{1}{2} \Phi(x)}
e^{\frac{1}{2} \Phi(x)} \| \gamma(h_1,x) - \gamma(h_2,x) \|_{\beta'}
F(dx) \bigg)^2
\\ & \qquad \qquad \times e^{\beta \xi} d\xi
\\ &\leq C_3^2 L^2 \| h_1 - h_2 \|_{\beta}^2 \int_E e^{\Phi(x)} \int_{\mathbb{R}_+} \bigg| \frac{d}{d \xi}
\gamma(h_1,x)(\xi) \bigg|^2 e^{\beta \xi} d\xi F(dx)
\\ &\leq C_3^2 L^2 \| h_1 - h_2 \|_{\beta}^2 \int_E e^{\Phi(x)} \| \gamma(h_1,x) \|_{\beta'}^2
F(dx) \leq C_3^2 L^2 M \| h_1 - h_2 \|_{\beta}^2.
\end{align*}
We obtain by (\ref{est-Gamma}), H\"older's inequality,
(\ref{bounded-gamma-HJM}), (\ref{est-embedding}) and
(\ref{Lipschitz-gamma-HJM})
\begin{align*}
I_4 &\leq C_3^2 \int_{\mathbb{R}_+} \bigg( \int_E \| \gamma(h_2,x)
\|_{\beta'} e^{\frac{1}{2}\Phi(x)}  e^{\frac{1}{2}\Phi(x)} \bigg|
\frac{d}{d\xi} \gamma(h_1,x)(\xi) - \frac{d}{d\xi}
\gamma(h_2,x)(\xi) \bigg| F(dx) \bigg)^2
\\ & \qquad \qquad \times e^{\beta \xi} d\xi
\\ &\leq C_3^2 M \int_E e^{\Phi(x)} \int_{\mathbb{R}_+} \bigg|
\frac{d}{d\xi} \gamma(h_1,x)(\xi) - \frac{d}{d\xi}
\gamma(h_2,x)(\xi) \bigg|^2 e^{\beta \xi} d\xi F(dx)
\\ &\leq C_3^2 M \int_E e^{\Phi(x)} \| \gamma(h_1,x) - \gamma(h_2,x)
\|_{\beta'}^2 F(dx) \leq C_3^2 M L^2 \| h_1 - h_2 \|_{\beta}^2.
\end{align*}
Summing up, we deduce that there is a constant $K > 0$ such that
(\ref{Lipschitz-desired}) is satisfied for all $h_1,h_2 \in
H_{\beta}$.
\end{proof}

\begin{theorem}\label{thm-ex-hjm}
Suppose Assumption \ref{ass-Lipschitz} is fulfilled. Then, for each
initial curve $h_0 \in L^2(\Omega,\mathcal{F}_0,\mathbb{P};H_{\beta})$
there exists a unique adapted, c\`adl\`ag, mean square continuous
$\mathcal{H}_{\beta}$-valued solution $(f_t)_{t \geq 0}$ for the HJM
equation (\ref{HJM-eqn}) with $f_0 = \ell h_0$ satisfying
\begin{align}\label{solution-in-S2-f}
\mathbb{E} \bigg[ \sup_{t \in [0,T]} \| f_t \|_{\beta}^2 \bigg] <
\infty \quad \text{for all $T \in \mathbb{R}_+$,}
\end{align}
and there exists a unique adapted, c\`adl\`ag, mean square
continuous mild and weak $H_{\beta}$-valued solution $(r_t)_{t \geq
0}$ for the HJMM equation (\ref{HJMM-eqn}) with $r_0 = h_0$
satisfying
\begin{align}\label{solution-in-S2}
\mathbb{E} \bigg[ \sup_{t \in [0,T]} \| r_t \|_{\beta}^2 \bigg] <
\infty \quad \text{for all $T \in \mathbb{R}_+$,}
\end{align}
which is given by $r_t := \pi U_t f_t$, $t \geq 0$. Moreover, the
implied bond market (\ref{bond-market}) is free of arbitrage.
\end{theorem}

\begin{proof}
By virtue of Theorem \ref{thm-group}, (\ref{Lipschitz-sigma-HJM}),
(\ref{Lipschitz-gamma-HJM}), (\ref{bounded-gamma-HJM}),
(\ref{est-embedding}) and (\ref{Lipschitz-desired}), the Assumptions
\ref{ass-group}, \ref{ass-1-mild-hom}, \ref{ass-2-mild-hom} are
fulfilled. Theorem \ref{thm-ex-SDE} applies and establishes the
claimed existence and uniqueness result.

For all $h \in H_{\beta}$, $x \in E$ and $\xi \in \mathbb{R}_+$ we
have by (\ref{est-Gamma-f}), (\ref{est-h}) and (\ref{est-embedding})
\begin{equation}\label{est-Gamma-second-order}
\begin{aligned}
&|e^{\Gamma(h,x)(\xi)} - 1 - \Gamma(h,x)(\xi)| \leq \frac{1}{2}
e^{\Phi(x)} \Gamma(h,x)(\xi)^2
\\ &\leq \frac{1}{2} e^{\Phi(x)} \| \gamma(h,x)
\|_{L^1(\mathbb{R}_+)}^2 \leq \frac{C_3^2}{2} e^{\Phi(x)} \|
\gamma(h,x) \|_{\beta'}^2.
\end{aligned}
\end{equation}
Integrating (\ref{def-alpha-HJM-intro}) we obtain, by using
\cite[Lemma 4.3.2]{fillnm} and (\ref{est-Gamma-second-order}),
(\ref{bounded-gamma-HJM})
\begin{align*}
\int_0^{\bullet} \alpha_{\rm HJM}(h)(\eta)d\eta = \frac{1}{2} \sum_j
\Sigma^j(h)^2 + \int_E \left( e^{\Gamma(h,x)} - 1 - \Gamma(h,x)
\right) F(dx)
\end{align*}
for all $h \in H_{\beta}$. Combining \cite[Prop. 5.3]{BKR} and
\cite[Lemma 4.3.3]{fillnm} (the latter result is only required if
$W$ is infinite dimensional), the probability measure $\mathbb{P}$
is a local martingale measure, and hence the bond market is free of
arbitrage.
\end{proof}

The case of L\'evy-driven HJMM equation is now a special case. We assume that the mark space
is $E = \mathbb{R}^e$ for some $e \in \mathbb{N}$, equipped with its
Borel $\sigma$-algebra $\mathcal{E} = \mathcal{B}(\mathbb{R}^e)$.
The measure $F$ is given by
\begin{align}\label{def-Levy-measure}
F(B) := \sum_{k=1}^e \int_{\mathbb{R}} \mathbbm{1}_B(x f_k) F_k(dx),
\quad B \in \mathcal{B}(\mathbb{R}^e)
\end{align}
where $F_1,\ldots,F_e$ are measures on
$(\mathbb{R},\mathcal{B}(\mathbb{R}))$ satisfying
\begin{align}\label{F-zero-in-zero}
&F_k(\{ 0 \}) = 0, \quad \quad k = 1,\ldots,e
\\ \label{F-standard-cond} &\int_{\mathbb{R}} (| x |^2 \wedge 1) F_k(dx)
< \infty, \quad k = 1,\ldots,e
\end{align}
and where the $(f_k)_{k=1,\ldots,e}$ denote the unit vectors in
$\mathbb{R}^e$. Note that definition (\ref{def-Levy-measure})
implies
\begin{align}\label{Levy-measures}
\int_{\mathbb{R}^e} g(x) F(dx) = \sum_{k=1}^e \int_{\mathbb{R}} g(x
f_k) F_k(dx)
\end{align}
for any nonnegative measurable function $g : \mathbb{R}^e
\rightarrow \mathbb{R}$, in particular, the support of $F$ is
contained in $\bigcup_{k=1}^e {\rm span} \{ f_k \}$, the union of
the coordinate axes in $\mathbb{R}^e$. For each $k = 1,\ldots,e$ let
$\delta_k : H_{\beta} \rightarrow H_{\beta'}^0$ be a map. We define
$\gamma : H_{\beta} \times \mathbb{R}^e \rightarrow H_{\beta'}^0$ as
\begin{align}\label{def-gamma-Levy}
\gamma(h,x) := \sum_{k=1}^e \delta_k(h) x_k \mathbbm{1}_{{\rm span}
\{ f_k \}}(x).
\end{align}
Then, equation (\ref{HJMM-eqn}) corresponds to the situation where
the term structure model is driven by several real-valued,
independent L\'evy processes. For all $h \in H_{\beta}$ and $\xi \in
\mathbb{R}_+$ we set $\Delta_k(h)(\xi) := - \int_0^{\xi}
\delta_k(h)(\eta) d\eta$, $k = 1,\ldots,e$.

\begin{assumption}\label{ass-cor}
We assume there exist constants $N,\epsilon > 0$ such that for all
$k = 1,\ldots,e$ we have
\begin{align}\label{F-exp}
&\int_{\{ | x | > 1 \}} e^{zx} F_k(dx) < \infty, \quad z \in
[-(1+\epsilon)N,(1+\epsilon)N]
\\ \label{domain} & | \Delta_k(h)(\xi) | \leq N,
\quad h \in H_{\beta}, \, \xi \in \mathbb{R}_+
\end{align}
a constant $L > 0$ such that (\ref{Lipschitz-sigma-HJM}) and
\begin{align}\label{Lipschitz-delta-HJM}
\| \delta_k(h_1) - \delta_k(h_2) \|_{\beta'} \leq L \| h_1 - h_2
\|_{\beta}, \quad k = 1,\ldots,e
\end{align}
are satisfied for all $h_1,h_2 \in H_{\beta}$, and a constant $M >
0$ such that (\ref{bounded-sigma-HJM}) and
\begin{align}\label{bounded-delta-HJM}
\| \delta_k(h) \|_{\beta'} \leq M, \quad k = 1,\ldots,e
\end{align}
are satisfied for all $h \in H_{\beta}$.
\end{assumption}

Now, we obtain the statement of \cite[Thm. 4.6]{Filipovic-Tappe} as
a corollary.

\begin{corollary}\label{cor-ex-HJM}
Suppose Assumption \ref{ass-cor} is fulfilled. Then, for each
initial curve $h_0 \in L^2(\Omega,\mathcal{F}_0,\mathbb{P};H_{\beta})$
there exists a unique adapted, c\`adl\`ag, mean square continuous
$\mathcal{H}_{\beta}$-valued solution $(f_t)_{t \geq 0}$ for the HJM
equation (\ref{HJM-eqn}) with $f_0 = \ell h_0$ satisfying
(\ref{solution-in-S2-f}), and there exists a unique adapted,
c\`adl\`ag, mean square continuous mild and weak $H_{\beta}$-valued
solution $(r_t)_{t \geq 0}$ for the HJMM equation (\ref{HJMM-eqn})
with $r_0 = h_0$ satisfying (\ref{solution-in-S2}), which is given
by $r_t := \pi U_t f_t$, $t \geq 0$. Moreover, the implied bond
market (\ref{bond-market}) is free of arbitrage.
\end{corollary}

\begin{proof}
Using (\ref{domain}), the measurable function $\Phi : \mathbb{R}^e
\rightarrow \mathbb{R}_+$ defined as
\begin{align*}
\Phi(x) := N \sum_{k=1}^e |x_k| \mathbbm{1}_{{\rm span}\{ f_k
\}}(x), \quad x \in \mathbb{R}^e
\end{align*}
satisfies (\ref{est-Gamma-f}). For each $k = 1,\ldots,e$ and every
$m \in \mathbb{N}$ with $m \geq 2$ we have, by
(\ref{F-standard-cond}) and (\ref{F-exp}),
\begin{equation}\label{Levy-estimate}
\begin{aligned}
&\int_{\mathbb{R}} |x^m e^{zx}| F_k(dx) \leq \int_{\mathbb{R}} |x|^m
e^{|zx|} F_k(dx) \leq \int_{\mathbb{R}} |x|^m
e^{(1+\frac{\epsilon}{2})N|x|} F_k(dx)
\\ &\leq 2 \int_{\{ |x| \leq \frac{\ln 2}{(1+\frac{\epsilon}{2})N} \}} |x|^m F_k(dx)
+ \frac{m!}{(\frac{\epsilon}{2} N)^m} \int_{\{ |x| > \frac{\ln
2}{(1+\frac{\epsilon}{2})N} \}} e^{(1 + \epsilon)N |x|} F_k(dx) <
\infty
\end{aligned}
\end{equation}
for all $z \in [-(1+\frac{\epsilon}{2})N,(1+\frac{\epsilon}{2})N]$.
Taking into account (\ref{Levy-measures}), (\ref{Levy-estimate}),
relations (\ref{Lipschitz-delta-HJM}), (\ref{bounded-delta-HJM})
imply (\ref{Lipschitz-gamma-HJM}), (\ref{bounded-gamma-HJM}).
Furthermore, (\ref{Levy-estimate}) and Lebesgue's theorem show that
the cumulant generating functions
\begin{align*}
\Psi_k(z) := \int_{\mathbb{R}} ( e^{zx} - 1 - zx ) F_k(dx), \quad k
= 1,\ldots,e
\end{align*}
belong to class $C^{\infty}$ on the open interval
$(-(1+\frac{\epsilon}{4})N,(1+\frac{\epsilon}{4})N)$ with
derivatives
\begin{align*}
\Psi_k'(z) &= \int_{\mathbb{R}} x ( e^{zx} - 1 ) F_k(dx),
\\ \Psi_k^{(m)}(z) &= \int_{\mathbb{R}} x^m e^{zx} F_k(dx), \quad m
\geq 2.
\end{align*}
Therefore, and because of (\ref{Levy-measures}), we can, for an
arbitrary $h \in H_{\beta}$, write $\alpha_2(h)$, which is defined
in (\ref{def-alpha-2}), as
\begin{align*}
\alpha_2(h) = - \sum_{k=1}^e \delta_k(h) \Psi_k' \bigg( -
\int_0^{\bullet} \delta_k(\eta) d\eta \bigg).
\end{align*}
Hence, $\alpha_2(h)$ is absolutely continuous with weak derivative
(\ref{der-alpha-2}). Consequently, Assumption \ref{ass-Lipschitz} is
fulfilled and Theorem \ref{thm-ex-hjm} applies.
\end{proof}

Note that the boundedness assumptions (\ref{bounded-sigma-HJM}),
(\ref{bounded-gamma-HJM}) of Theorem \ref{thm-ex-hjm} resp.
(\ref{bounded-sigma-HJM}), (\ref{bounded-delta-HJM}) of Corollary
\ref{cor-ex-HJM} cannot be weakened substantially. For example, for
arbitrage free term structure models driven by a single Brownian
motion, it was shown in \cite[Sec. 4.7]{Morton} that for the simple
case of proportional volatility, that is $\sigma(h) = \sigma_0 h$
for some constant $\sigma_0 > 0$, solutions necessarily explode. We
mention, however, that \cite[Sec. 6]{P-Z-paper} contains some
existence results for L\'evy term structure models with linear
volatility.

\section{Positivity preserving term structure models driven by Wiener process and Poisson
measures}\label{sec-positivity}

In applications, we are often interested in term structure models producing
positive forward curves. In this section, we characterize HJMM
forward curve evolutions of the type (\ref{HJMM-eqn}), which
preserve positivity, by means of the characteristic coefficients of the SPDE. In the case of short rate models
this can be characterized by the positivity of the short rate, a one-dimensional Markov process. In case of a infinite-factor evolution, as described by a generic HJMM equation (see for instance \cite{bautei:05}), this problem is much more delicate. Indeed, one has to find conditions such that a Markov process defined by the HJMM equation (on a Hilbert space of forward rate curves) stays in a ``small'' set of curves, namely the convex cone of positive curves bounded by a non-smooth set. Our strategy to solve this problem is the following: first we show by general semimartingale methods necessary conditions for positivity. These necessary conditions are basically the described by the facts that the It\^o drift is inward-pointing and that the volatilities are parallel at the boundary of the set of non-negative functions. Taking those conditions we can also prove that the Stratonovich drift is inward pointing, since parallel volatilities produce parallel Stratonovich corrections (a fact which is not true for general closed convex sets but true for the set of non-negative functions $P$). Then we reduce the sufficiency proof to two steps: first we essentially apply results from \cite{Nakayama} in order to solve the pure diffusion case and then we slowly switch on the jumps to see the general result.

Let $H_{\beta}$ be the space of forward curves introduced in Section
\ref{sec-space} for some fixed $\beta > 0$. We introduce the half
spaces
\begin{align*}
H_{\xi}^+ := \{ h \in H_{\beta} \,|\, h(\xi) \geq 0 \}, \quad \xi
\in \mathbb{R}_+
\end{align*}
and define the closed, convex cone
\begin{align*}
P := \bigcap_{\xi \in \mathbb{R}_+} H_{\xi}^+
\end{align*}
consisting of all nonnegative forward curves from $H_{\beta}$. In
what follows, we shall use that, by the continuity of the functions
from $H_{\beta}$, we can write $P$ as
\begin{align*}
P := \bigcap_{\xi \in (0,\infty)} H_{\xi}^+.
\end{align*}
Furthermore the edges
%, we define the reflected half spaces
%\begin{align*}
%H_{\xi}^- := \{ h \in H_{\beta} \,|\, h(\xi) \leq 0 \}, \quad \xi
%\in (0,\infty)
%\end{align*}
%and 
\begin{align*}
\partial P_{\xi} := \{ h \in P \,|\, h(\xi) = 0 \}, \quad \xi \in (0,\infty).
\end{align*}
First, we consider the positivity problem for general forward curve
evolutions, where the HJM drift condition
(\ref{def-alpha-HJM-intro}) is not necessarily satisfied, and
afterwards we apply our results to the arbitrage free situation.

We emphasize that, in the sequel, we \textit{assume} the existence
of solutions. Sufficient conditions for existence and uniqueness are provided in
Appendix \ref{app-SDE} for general stochastic partial differential
equations and in the previous Section \ref{sec-ex-tsm} for the HJMM
term structure equation (\ref{HJMM-eqn}).

Again, the framework is the same as in Appendix \ref{app-SDE} with
$H = H_{\beta}$ being the space of forward curves, equipped with the
strongly continuous semigroup $(S_t)_{t \geq 0}$ of shifts, which
has the infinitesimal generator $A = \frac{d}{d \xi}$.

At first glance, it looks reasonable to treat the positivity problem
by working with \textit{weak} solutions on $H_{\beta}$. However,
this is unfeasible, because -- as the next lemma reveals -- the point evaluations at $\xi
\in (0,\infty)$, i.e., a linear functional $\zeta \in H_{\beta}$ such
that $h(\xi) = \langle \zeta, h \rangle$ for all $h \in H_{\beta}$, do never belong
to the domain $\mathcal{D}((\frac{d}{d\xi})^*)$ of the adjoint operator.

\begin{lemma}
For each $\xi \in (0,\infty)$ the linear functional $h \mapsto
h'(\xi) : \mathcal{D}(\frac{d}{d\xi}) \rightarrow
\mathbb{R}$ is unbounded.
\end{lemma}

\begin{proof}
Let $\xi \in (0,\infty)$ be arbitrary. We define $\psi : \mathbb{R}
\rightarrow \mathbb{R}$ as
\begin{align*}
\psi(\eta) :=
\begin{cases}
e^{-\frac{1}{\eta}}, & \eta > 0
\\ 0, & \eta \leq 0.
\end{cases}
\end{align*}
Furthermore, we define $\varphi : \mathbb{R} \rightarrow \mathbb{R}$
as $\varphi(\eta) := e \psi(1 - \eta^2)$, and for each $n \in
\mathbb{N}$ we define the mollifier $\varphi_n : \mathbb{R}
\rightarrow \mathbb{R}$ by $\varphi_n(\eta) := n \varphi(n \eta)$.
Then, for each $n \in \mathbb{N}$ we have $\varphi_n \in
C^{\infty}(\mathbb{R})$ with ${\rm supp}(\varphi_n) \subset
[-\frac{1}{n},\frac{1}{n}]$, see, e.g., \cite[p. 81,82]{Werner}.
There exists $n_0 \in \mathbb{N}$ such that $\xi - \frac{1}{n} > 0$
for all $n \geq n_0$. For each $n \geq n_0$ we define $g_n :
\mathbb{R} \rightarrow \mathbb{R}$ as $g_n(\eta) := \int_0^{\eta}
\varphi_n(\zeta - \xi)d\zeta$ and $h_n : \mathbb{R}_+ \rightarrow
\mathbb{R}$ as $h_n := g_n|_{\mathbb{R}_+}$.

Then we have $h_n \in \mathcal{D}(\frac{d}{d\xi})$ with
$|h_n'(\eta)| \leq n$, $\eta \in [\xi - \frac{1}{n}, \xi +
\frac{1}{n}]$ and $|h_n'(\xi)| = n$ for all $n \geq n_0$. The
estimate
\begin{align*}
\| h_n \|_{\beta}^2 = \int_{\xi - \frac{1}{n}}^{\xi + \frac{1}{n}}
|h_n'(\eta)|^2 e^{\beta \eta} d\eta \leq \frac{2}{n} n^2
e^{\beta(\xi + 1)} = 2 e^{\beta(\xi + 1)} n, \quad n \geq n_0
\end{align*}
shows that the linear functional $h \mapsto h'(\xi) :
\mathcal{D}(\frac{d}{d\xi}) \rightarrow \mathbb{R}$ is
unbounded.
\end{proof}

Therefore treating the positivity problem with weak solutions does not bring an immediate advantage, hence we shall work with \textit{mild} solutions on $H_{\beta}$.

Let $\alpha : H_{\beta} \rightarrow H_{\beta}$, $\sigma : H_{\beta}
\rightarrow L_2^0(H_{\beta})$ and $\gamma : H_{\beta} \times E
\rightarrow H_{\beta}$ be given. For each $j$ we define $\sigma^j :
H_{\beta} \rightarrow H_{\beta}$ as $\sigma^j(h) := \sqrt{\lambda_j}
\sigma(h) e_j$. We assume that for each $h_0 \in P$ the HJM equation
\begin{align}\label{HJM-eqn-new}
\left\{
\begin{array}{rcl}
df_t & = & U_{-t} \ell \alpha(\pi U_t f_t)dt + \sum_j U_{-t} \ell
\sigma^j(\pi U_t f_t) d\beta_t^j
\\ && + \int_E U_{-t}
\ell \gamma(\pi U_t f_{t-},x) (\mu(dt,dx) - F(dx)dt)
\medskip
\\ f_0 & = & \ell h_0,
\end{array}
\right.
\end{align}
has at least one solution $(f_t)_{t \geq 0}$. Then, because of
(\ref{diagram-commutes}), the transformation $r_t := \pi U_t f_t$,
$t \geq 0$ is a mild solution of the HJMM equation
\begin{align}\label{equation-new}
\left\{
\begin{array}{rcl}
dr_t & = & (\frac{d}{d\xi} r_t + \alpha(r_t))dt + \sum_{j}
\sigma^j(r_t) d\beta_t^j + \int_E \gamma(r_{t-},x) (\mu(dt,dx) -
F(dx)dt)
\medskip
\\ r_0 & = & h_0.
\end{array}
\right.
\end{align}

\begin{definition}
The HJMM equation (\ref{equation-new}) is said to be {\rm positivity
preserving} if for all $h_0 \in
L^2(\Omega,\mathcal{F}_0,\mathbb{P};H_{\beta})$ with $\mathbb{P}(h_0
\in P) = 1$ and every solution $(f_t)_{t \geq 0}$ of
(\ref{HJM-eqn-new}) with $f_0 = \ell h_0$ we have $\mathbb{P}(
\bigcap_{t \in \mathbb{R}_+} \{ r_t \in P \} ) = 1$, where $r_t :=
\pi U_t f_t$, $t \geq 0$.
\end{definition}

\begin{remark}
Note that seemingly weaker condition that $\mathbb{P}( \{ r_t \in P
\} ) = 1$ for all $ t \in \mathbb{R}_+ $ is equivalent to condition
of the previous definition due to the c\`{a}dl\`{a}g property of the
trajectories.
\end{remark}

\begin{definition}
The HJMM equation (\ref{equation-new}) is said to be {\rm locally
positivity preserving} if for all $h_0 \in
L^2(\Omega,\mathcal{F}_0,\mathbb{P};H_{\beta})$ with $\mathbb{P}(h_0
\in P) = 1$ and every solution $(f_t)_{t \geq 0}$ of
(\ref{HJM-eqn-new}) with $f_0 = \ell h_0$ there exists a strictly
positive stopping time $\tau$ such that $\mathbb{P}( \bigcap_{t \in
\mathbb{R}_+} \{ r_{t \wedge \tau} \in P \} ) = 1$, where $r_t :=
\pi U_t f_t$, $t \geq 0$.
\end{definition}

\begin{lemma}\label{lemma-positivity}
Let $h_0 \in P$ be arbitrary and let $(f_t)_{t \geq 0}$ be a
solution for (\ref{HJM-eqn-new}) with $f_0 = \ell h_0$. Set $r_t :=
\pi U_t f_t$, $t \geq 0$. The following two statements are
equivalent:
\begin{enumerate}
\item We have $\mathbb{P}( \bigcap_{t \in \mathbb{R}_+} \{ r_t \in P \} ) =
1$.

\item We have $\mathbb{P}( \bigcap_{t \in [0,T]} \{ f_t(T) \geq 0 \} ) =
1$ for all $T \in (0,\infty)$.
\end{enumerate}
\end{lemma}

\begin{proof}
The claim follows, because the processes $(r_t)_{t \geq 0}$ and
$(f_t(T))_{t \in [0,T]}$ for an arbitrary $T \in (0,\infty)$ are
c\`adl\`ag, and because the functions from $H_{\beta}$ are
continuous.
\end{proof}

\begin{assumption}\label{ass-nec}
We assume that the maps $\alpha : H_{\beta} \rightarrow H_{\beta}$
and $\sigma : H_{\beta} \rightarrow L_2^0(H_{\beta})$ are continuous
and that $h \mapsto \int_{B} \gamma(h,x) F(dx)$ is continuous on
$H_{\beta}$ for all $B \in \mathcal{E}$ with $F(B) < \infty$.
\end{assumption}

\begin{proposition}\label{prop-pos-general-nec}
Suppose Assumption \ref{ass-nec} is fulfilled. If equation
(\ref{equation-new}) is positivity preserving, then we have
\begin{align}\label{cond-integrable-pre}
&\int_{E} \gamma(h,x)(\xi) F(dx) < \infty, \text{ for all } \xi \in (0,\infty),
\, h \in
\partial P_{\xi}
\\ \label{cond-alpha-pre} &\alpha(h)(\xi) - \int_{E}
\gamma(h,x)(\xi) F(dx) \geq 0, \text{ for all } \xi \in (0,\infty), \, h \in
\partial P_{\xi}
\\ \label{cond-sigma-pre} &\sigma^j(h)(\xi) = 0, \text{ for all } \xi \in (0,\infty),\, h \in
\partial P_{\xi} \text{ and all $j$}
\\ \label{cond-gamma-pre} &h + \gamma(h,x) \in P, \text{ for all } \, h \in P \text{ and $F$-almost all $x \in E$.}
\end{align}
\end{proposition}

\begin{remark}\label{remark-continuity}
Notice that, by H\"older's inequality, Assumption \ref{ass-nec} is
implied by Assumptions \ref{ass-1-mild-hom}, \ref{ass-2-mild-hom},
and therefore in particular by Assumption \ref{ass-Lipschitz}.
\end{remark}

\begin{remark}
In view of (\ref{cond-integrable-pre}), observe that condition
(\ref{cond-gamma-pre}) implies
\begin{align}\label{observation}
\gamma(h,x)(\xi) \geq 0, \text{ for all } \xi \in (0,\infty), \, h \in
\partial P_{\xi} \text{ and $F$-almost all $x \in E$.}
\end{align}
\end{remark}

\begin{remark}
Notice that conditions \eqref{cond-integrable-pre} and \eqref{cond-alpha-pre} can be unified to
$$
\int_{E} |\gamma(h,x)(\xi) | F(dx) \leq \alpha(h)(\xi)
$$
for all $ \xi \geq 0 $ and $ h \in \partial P_{\xi} $.
\end{remark}

\begin{proof}
Let $h_0 \in P$ be arbitrary and let $(f_t)_{t \geq 0}$ be a
solution for (\ref{HJM-eqn-new}) with $f_0 = \ell h_0$. By Lemma
\ref{lemma-positivity}, for each $T \in (0,\infty)$ and every
stopping time $\tau \leq T$ we have
\begin{align}\label{pos-contr}
\mathbb{P}(f_{\tau}(T) \geq 0) = 1.
\end{align}
Let $\phi \in U_0'$ be a linear functional such that $\phi^j := \phi
e_j \neq 0$ for only finitely many $j$, and let $\psi : E
\rightarrow \mathbb{R}$ be a measurable function of the form $\psi =
c \mathbbm{1}_B$ with $c > -1$ and $B \in \mathcal{E}$ satisfying
$F(B) < \infty$. Let $Z$ be the Dol\'eans-Dade Exponential
\begin{align*}
Z_t = \mathcal{E} \bigg( \sum_{j} \phi^j \beta^j + \int_0^{\bullet}
\int_{E} \psi(x) (\mu(ds,dx) - F(dx)ds) \bigg)_t, \quad t \geq 0.
\end{align*}
By \cite[Thm. I.4.61]{Jacod-Shiryaev} the process $Z$ is a solution
of
\begin{align*}
Z_t = 1 + \sum_{j} \phi^j \int_0^t Z_s d\beta_s^j + \int_0^t
\int_{E} Z_{s-} \psi(x) (\mu(ds,dx) - F(dx)ds), \quad t \geq 0
\end{align*}
and, since $\psi > -1$, the process $Z$ is a strictly positive local
martingale. There exists a strictly positive stopping time $\tau_1$
such that $Z^{\tau_1}$ is a martingale. Due to the method of the moving frame, see \cite{SPDE}, we can use
standard stochastic analysis, to proceed further. For an arbitrary $T \in
(0,\infty)$, integration by parts yields (see \cite[Thm.
I.4.52]{Jacod-Shiryaev})
\begin{equation}\label{product}
\begin{aligned}
f_t(T) Z_t &= \int_0^t f_{s-}(T) dZ_s + \int_0^t Z_{s-} df_s(T) +
\langle f(T)^{c},Z^c \rangle_t
\\ &\quad + \sum_{s \leq t} \Delta f_s(T) \Delta Z_s, \quad t \geq 0.
\end{aligned}
\end{equation}
Taking into account the dynamics
\begin{equation}\label{dynamics-R}
\begin{aligned}
f_t(T) &= \ell h_0(T) + \int_0^t U_{-s} \ell \alpha(\pi U_s f_s)(T)
ds + \sum_{j} \int_0^t U_{-s} \ell \sigma^j(\pi U_s f_s)(T)
d\beta_s^j
\\ &\quad + \int_0^t \int_{E} U_{-s} \ell
\gamma(\pi U_s f_{s-},x)(T) (\mu(ds,dx) - F(dx)ds), \quad t \geq 0
\end{aligned}
\end{equation}
we have
\begin{align}\label{rhs0}
\langle f(T)^{c}, Z^c \rangle_t &= \sum_{j} \phi^j \int_0^t Z_s
U_{-s} \ell \sigma^j(\pi U_s f_s)(T) ds, \quad t \geq 0
\\ \label{rhs} \sum_{s \leq t} \Delta f_s(T) \Delta Z_s &= \int_0^t
\int_{E} Z_{s-} \psi(x) U_{-s} \ell \gamma(\pi U_s f_{s-},x)(T)
\mu(ds,dx), \quad t \geq 0.
\end{align}
Incorporating (\ref{dynamics-R}), (\ref{rhs0}) and (\ref{rhs}) into
(\ref{product}), we obtain
\begin{equation}\label{make-contradiction}
\begin{aligned}
f_t(T) Z_t &= M_t + \int_0^t Z_{s-} \bigg( U_{-s} \ell \alpha(\pi
U_s f_{s-})(T) + \sum_{j} \phi^j U_{-s} \ell \sigma^j(\pi U_s
f_{s-})(T)
\\ &\quad + \int_{E}
\psi(x) U_{-s} \ell \gamma(\pi U_s f_{s-},x)(T) F(dx) \bigg) ds,
\quad t \geq 0
\end{aligned}
\end{equation}
where $M$ is a local martingale with $M_0 = 0$. There exists a
strictly positive stopping time $\tau_2$ such that $M^{\tau_2}$ is a
martingale.

By Assumption \ref{ass-nec} there exist strictly positive stopping
times $\tau_3,\tau_4,\tau_5$ and constants
$\tilde{\alpha},\tilde{\sigma}(\phi),\tilde{\gamma}(\psi) > 0$ such
that
\begin{align*}
&|U_{-(t \wedge \tau_3)} \ell \alpha(\pi U_{t \wedge \tau_3} f_{(t
\wedge \tau_3)-})(T) | \leq \tilde{\alpha}, \quad t \geq 0
\\ &\bigg| \sum_j \phi^j U_{-(t \wedge \tau_4)} \ell \sigma^j(\pi U_{t
\wedge \tau_4} f_{(t \wedge \tau_4)-})(T) \bigg| \leq
\tilde{\sigma}(\phi), \quad t \geq 0
\\&\bigg| \int_{E} \psi(x) U_{-(t \wedge \tau_5)} \ell \gamma(\pi U_{t
\wedge \tau_5} f_{(t \wedge \tau_5)-},x)(T) F(dx) \bigg| \leq
\tilde{\gamma}(\psi), \quad t \geq 0.
\end{align*}
Let $B := \{ x \in E : h_0 + \gamma(h_0,x) \notin P \}$. In order to
prove (\ref{cond-gamma-pre}), it suffices, since $F$ is
$\sigma$-finite, to show that $F(B \cap C) = 0$ for all $C \in
\mathcal{E}$ with $F(C) < \infty$. Suppose, on the contrary, there
exists $C \in \mathcal{E}$ with $F(C) < \infty$ such that $F(B \cap
C) > 0$. By the continuity of the functions from $H_{\beta}$, there
exists $T \in (0,\infty)$ such that $F(B_T \cap C) > 0$, where
$B_{T} := \{ x \in E : h_0(T) + \gamma(h_0,x)(T) < 0 \}$. We obtain
\begin{align*}
\int_{B_T \cap C} \gamma(h_0,x)(T) F(dx) \leq \int_{B_T \cap C}
(h_0(T) + \gamma(h_0,x)(T)) F(dx) < 0.
\end{align*}
By Assumption \ref{ass-nec} and left continuity of the process $ f_{.-} $, there exist $\eta > 0$ and a strictly
positive stopping time $\tau_6 \leq T$ such that
\begin{align*}
\int_{B_T \cap C} U_{-(t \wedge \tau_6)} \ell \gamma(\pi U_{(t
\wedge \tau_6)} f_{(t \wedge \tau_6)-},x)(T) F(dx) \leq - \eta,
\quad t \geq 0.
\end{align*}
Let $\phi := 0$, $\psi := \frac{\tilde{\alpha} + 1}{\eta}
\mathbbm{1}_{B_T \cap C}$ and $\tau := \bigwedge_{i=1}^6 \tau_i$.
Taking expectation in (\ref{make-contradiction}) we obtain
$\mathbb{E}[f_{\tau}(T) Z_{\tau}] < 0$, implying
$\mathbb{P}(f_{\tau}(T) < 0) > 0$, which contradicts
(\ref{pos-contr}). This yields (\ref{cond-gamma-pre}).

From now on, we assume that $h_0 \in \partial P_{T}$ for an
arbitrary $T \in (0,\infty)$.

Suppose that $\sigma^j(h_0)(T) \neq 0$ for some $j$. By the
continuity of $\sigma$ (see Assumption \ref{ass-nec}) there exist
$\eta > 0$ and a strictly positive stopping time $\tau_6 \leq T$
such that
\begin{align*}
|U_{-(t \wedge \tau_6)} \ell \sigma^{j}(\pi U_{t \wedge \tau_6}
f_{(t \wedge \tau_6)-})(T)| \geq \eta, \quad t \geq 0.
\end{align*}
Let $\phi \in U_0'$ be the linear functional with $\phi^j = -{\rm
sign}(\sigma^{j}(h_0)(T)) \frac{\tilde{\alpha} + 1}{\eta}$ and
$\phi^k = 0$ for $k \neq j$. Furthermore, let $\psi := 0$ and $\tau
:= \bigwedge_{i=1}^6 \tau_i$. Taking expectation in
(\ref{make-contradiction}) yields $\mathbb{E}[f_{\tau}(T) Z_{\tau}]
< 0$, implying $\mathbb{P}(f_{\tau}(T) < 0) > 0$, which contradicts
(\ref{pos-contr}). This proves (\ref{cond-sigma-pre}).

Now suppose $\int_E \gamma(h_0,x)(T) F(dx) = \infty$. Using
Assumption \ref{ass-nec}, relation (\ref{observation}) and the
$\sigma$-finiteness of $F$, there exist $B \in \mathcal{E}$ with
$F(B) < \infty$ and a strictly positive stopping time $\tau_6 \leq
T$ such that
\begin{align*}
-\frac{1}{2} \int_{B} U_{-(t \wedge \tau_6)} \ell \gamma(\pi U_{t
\wedge \tau_6} f_{(t \wedge \tau_6)-},x)(T) F(dx) \leq
-(\tilde{\alpha}+1), \quad t \geq 0.
\end{align*}
Let $\phi := 0$, $\psi := -\frac{1}{2} \mathbbm{1}_{B}$ and $\tau :=
\bigwedge_{i=1}^6 \tau_i$. Taking expectation in
(\ref{make-contradiction}) we obtain $\mathbb{E}[f_{\tau}(T)
Z_{\tau}] < 0$, implying $\mathbb{P}(f_{\tau}(T) < 0) > 0$, which
contradicts (\ref{pos-contr}). This yields
(\ref{cond-integrable-pre}).

Since $F$ is $\sigma$-finite, there exists a sequence $(B_n)_{n \in
\mathbb{N}} \subset \mathcal{E}$ with $B_n \uparrow E$ and $F(B_n) <
\infty$, $n \in \mathbb{N}$. Next, we show for all $n \in
\mathbb{N}$ the relation
\begin{align}\label{cond-drift-approx}
\alpha(h_0)(T) + \int_E \psi_n(x) \gamma(h_0,x)(T) F(dx) \geq 0,
\end{align}
where $\psi_n := - (1 - \frac{1}{n}) \mathbbm{1}_{B_n}$. Suppose, on
the contrary, that (\ref{cond-drift-approx}) is not satisfied for
some $n \in \mathbb{N}$. Using Assumption \ref{ass-nec}, there exist
$\eta > 0$ and a strictly positive stopping time $\tau_6 \leq T$
such that
\begin{align*}
&U_{-(t \wedge \tau_6)} \ell \alpha(\pi U_t f_{(t \wedge
\tau_6)-})(T)
\\ &+ \int_E \psi_n(x) U_{-(t \wedge \tau_6)}
\ell \gamma(\pi U_{t \wedge \tau_6} f_{(t \wedge \tau_6)-},x)(T)
F(dx) \leq -\eta, \quad t \geq 0.
\end{align*}
Let $\phi := 0$ and $\tau := \bigwedge_{i=1}^6 \tau_i$. Taking
expectation in (\ref{make-contradiction}) we obtain
$\mathbb{E}[f_{\tau}(T) Z_{\tau}] < 0$, implying
$\mathbb{P}(f_{\tau}(T) < 0) > 0$, which contradicts
(\ref{pos-contr}). This yields (\ref{cond-drift-approx}). By
(\ref{cond-drift-approx}), (\ref{cond-integrable-pre}) and
Lebesgue's theorem, we conclude (\ref{cond-alpha-pre}).
\end{proof}

We shall now present sufficient conditions for positivity preserving
term structure models. In the sequel, we suppose that Assumptions
\ref{ass-1-mild-hom}, \ref{ass-2-mild-hom} are fulfilled, which
ensures existence and uniqueness of solutions by Theorem
\ref{thm-ex-SDE}.

\begin{lemma}\label{lemma-local-global}
Suppose Assumptions \ref{ass-1-mild-hom}, \ref{ass-2-mild-hom} are
fulfilled. If equation (\ref{equation-new}) is locally positivity
preserving and we have (\ref{cond-gamma-pre}), then equation
(\ref{equation-new}) is positivity preserving.
\end{lemma}

\begin{proof}
Let $h_0 \in L^2(\Omega,\mathcal{F}_0,\mathbb{P};H_{\beta})$ be
arbitrary. Moreover, let $(r_t)_{t \geq 0}$ be the mild solution for
(\ref{equation-new}) with $r_0 = h_0$. We define the stopping time
\begin{align}\label{def-tau-0}
\tau_0 := \inf \{ t > 0 : r_t \notin P \}.
\end{align}
By the closedness of $P$ and (\ref{cond-gamma-pre}) we have
$r_{\tau_0} \in P$ on $\{ \tau_0 < \infty \}$. We claim that $\tau_0
= \infty$. Assume, on the contrary, that
\begin{align}\label{pos-probability}
\mathbb{P}(\tau_0 < N) > 0
\end{align}
for some $N \in \mathbb{N}$. Let $\tau_1$ be the bounded stopping
time $\tau_1 := \tau_0 \wedge N$. We define the new filtration
$(\tilde{\mathcal{F}}_t)_{t \geq 0}$, the new $Q$-Wiener process
$\tilde{W}$ and the new Poisson random measure $\tilde{\mu}$ as in
Lemma \ref{lemma-solution-tau}. Note that $r_{\tau_1} \in
L^2(\Omega,\tilde{\mathcal{F}}_0,\mathbb{P};H_{\beta})$, because, by
(\ref{solution-in-S2}), we have
\begin{align*}
\mathbb{E}[\| r_{\tau_1} \|_{\beta}^2] \leq \mathbb{E} \bigg[
\sup_{t \in [0,N]} \| r_t \|_{\beta}^2 \bigg] < \infty.
\end{align*}
By Lemma \ref{lemma-solution-tau}, the
$(\tilde{\mathcal{F}}_t)$-adapted process $\tilde{r}_t := r_{\tau_1
+ t}$ is the unique mild solution for
\begin{align*}
\left\{
\begin{array}{rcl}
d\tilde{r}_t & = & (\frac{d}{d \xi} \tilde{r}_t +
\alpha(\tilde{r}_t)) dt + \sum_j
\sigma^j(\tilde{r}_t)d\tilde{\beta}_t^j + \int_E
\gamma(\tilde{r}_{t-},x)(\tilde{\mu}(dt,dx) - F(dx)dt)
\medskip
\\ \tilde{r}_0 & = & r_{\tau_1}.
\end{array}
\right.
\end{align*}
Since equation (\ref{equation-new}) is locally positivity preserving
and $\mathbb{P}(r_{\tau_1} \in P) = 1$, there exists a strictly
positive stopping time $\tau_2$ such that $\mathbb{P}( \bigcap_{t
\in \mathbb{R}_+} \tilde{r}_{t \wedge \tau_2} \in P) = 1$. Since $\{
\tau_0 < N \} \subset \{ \tau_0 = \tau_1 \}$, we obtain
\begin{align*}
r_{\tau_0 + t} \in P \quad \text{on $[0,\tau_2] \cap \{ \tau_0 <
N\}$,}
\end{align*}
which is a contradiction because of (\ref{pos-probability}) and the
definition (\ref{def-tau-0}) of $\tau_0$. Consequently, we have
$\tau_0 = \infty$.
\end{proof}

\begin{assumption}\label{ass-smooth}
We assume $\sigma \in C^2(H;L_2^0(H))$, and that the vector fields
$$
h \mapsto \sigma^0(h):= \alpha(h) - \frac{1}{2} \sum_{j} D \sigma^j(h) \sigma^j(h), \quad h \mapsto \sigma^j(h)
$$
are globally Lipschitz from $ H $ to $ H $.
\end{assumption}

\begin{lemma}\label{strat-drift_equals_drift}
Suppose Assumption \ref{ass-smooth} and relation
\eqref{cond-sigma-pre} are fulfilled. Then we have
\begin{align*}
\bigl ( \sum_j D \sigma^j(h) \sigma^j(h) \bigr)(\xi) = 0, \text{ for all } \xi \in
(0,\infty), \, h \in
\partial P_{\xi}.
\end{align*}
\end{lemma}

\begin{proof}
It suffices to show $ \bigl( D \sigma^j (h) \sigma^j(h) \bigr) (\xi) = 0 $ for all $ h \in \partial P_{\xi} $ and all $ j $. Therefore let $ j $ be fixed and denote $ \sigma=\sigma^j $. By assumption for all $ h \geq 0 $ with $ h(\xi)=0 $ we have that $ \sigma(h)(\xi) = 0 $. In other words the volatility vector field $ \sigma $ is parallel to the boundary at boundary elements of $ P $. We denote the local flow of the Lipschitz vector field $ \sigma $ by $ \operatorname{Fl} $ being defined on a small time interval $ ]-\epsilon, \epsilon[ $ around time $ 0 $ and a small neighborhood of each element $ h \in P $. We state first that the flow $ \operatorname{Fl} $ leaves the set $ P $ invariant, i.e., $ \operatorname{Fl}_t(h) \geq 0 $ if $ h \geq 0 $, by convexity and closedness of the cone of positive functions due to \cite{vol:73}. Indeed, $ P $ is a closed and convex cone, whose supporting hyperplanes $ l $ (i.e., a linear functional $ l $ is called supporting hyperplane of $ P $ at $ h $ if $ l(P) \geq 0 $ and $ l(h) = 0 $) are given by appropriate positive measures $ \mu $ on $ \mathbb{R}_+ $ via
$$
l(h) = \int_{\mathbb{R}_+} h(\xi) \mu(d \xi),
$$
whence condition (4) from \cite{vol:73} is fulfilled due to \eqref{cond-sigma-pre}. Next we show that even more holds: the solution $ \operatorname{Fl}_t(h) $ evaluated at $ \xi $ vanishes if $ h(\xi) = 0 $, which we show directly. Indeed, let us additionally fix $ h \in \partial P_{\xi} $, i.e., $ h \geq 0 $ and $ h(\xi) = 0 $. Looking now at the Picard-Lindel\"of approximation scheme
$$
c^{(n+1)}(t) = h + \int_0^t \sigma(c^{(n)}(s)) ds
$$
with $ c^{(n)}(0)=h $ and $ c^{(0)}(s) = h $ for $ s,t \in ]-\epsilon, \epsilon[ $ and $ n \geq 0 $, we see by induction that under our assumptions
$$
c^{(n)}(t)(\xi) = 0
$$
for all $ n \geq 0 $ and $ t \in ]-\epsilon,\epsilon[ $ for the given fixed element $ h $. Consequently -- as $ n \to \infty $ -- we obtain that $ \operatorname{Fl}_t(h)(\xi) = 0 $, which is the limit of $ c^{(n)}(t) $. Therefore
$$
(D \sigma(h) \sigma (h)) (\xi) = \frac{d}{ds}|_{s=0} \sigma(\operatorname{Fl}_s(h))(\xi) = 0,
$$
since $ \operatorname{Fl}_t(h) \geq 0 $ by invariance and $ \operatorname{Fl}_t(h) (\xi) = 0 $ by the previous consideration lead to $ \sigma(\operatorname{Fl}_t(h))(\xi)=0 $ for $ t \in ]-\epsilon, \epsilon [$. Notice that we did not need the global Lipschitz property of the Stratonovich correction for the proof of this lemma.
\end{proof}
Before we show sufficiency for the HJMM equation
(\ref{equation-new}) with jumps, we consider the
pure diffusion case. Notice that due to Lemma \ref{strat-drift_equals_drift} the condition (\ref{cond-alpha-pre}) is in fact equivalent to the very same condition formulated with the Stratonovich drift $\sigma^0$ instead of $\alpha$, since the Stratonovich correction vanishes at the boundary of $P$.

In order to treat the pure diffusion case, we apply \cite{Nakayama}, which, by using the support theorem provided in \cite{Nakayama-Support}, offers a general characterization of stochastic invariance of closed sets for SPDEs.

Other results for positivity preserving SPDEs, where, in contrast to our framework, the state space is an $L^2$-space, can be found in \cite{Kotelenez} and \cite{Milian-inf}. The results from \cite{Milian-inf} have been used in \cite{P-Z-paper} in order to derive some positivity results for L\'evy term structure models on $L^2$-spaces.

\begin{proposition}\label{prop-inv-Wiener}
Suppose Assumptions \ref{ass-2-mild-hom}, \ref{ass-smooth} are
fulfilled and $\gamma \equiv 0$. If conditions
(\ref{cond-alpha-pre}), (\ref{cond-sigma-pre}) are satisfied, then
equation (\ref{equation-new}) is positivity preserving.
\end{proposition}

\begin{proof}
First we assume that the vector fields $ \sigma^j $ for $ j \geq 0 $ are bounded in order to apply Nakayama's beautiful support theorem from \cite{Nakayama}. Namely, for $n \in \mathbb{N}$ we choose a function $\psi_n \in C^{\infty}(H;[0,1])$ such that $\psi_n \equiv 1$ on $\overline{B_n(0)}$ and ${\rm supp}(\psi_n) \subset B_{n+1}(0)$ and define the vector fields
\begin{align*}
\alpha_n(h) &:= \psi_n(h) \alpha(h),
\\ \sigma_n^j(h) &:= \psi_n(h) \sigma^j(h), \quad j \geq 1
\\ \sigma_n^0(h) &:= \psi_n(h) \alpha(h) - \frac{1}{2} \sum_j D \sigma_n^j(h) \sigma_n^j(h),
\end{align*}
which again satisfy Assumptions \ref{ass-2-mild-hom}, \ref{ass-smooth} as well as conditions
(\ref{cond-alpha-pre}), (\ref{cond-sigma-pre}).

We therefore show that the semigroup Nagumo's condition (3) from \cite[Prop. 1.1]{Nakayama-Support} is fulfilled due to conditions (\ref{cond-alpha-pre}) and (\ref{cond-sigma-pre}). Introducing the distance $ d_P(h) $ from $ P $ as minimal distance of $ h $ from $ P $, we can formulate Nagumo's condition (3) as
$$
\liminf_{t \downarrow 0} \frac{1}{t} d_P (S_t h + t \sigma^0(h) + t \sigma(h) u ) = 0
$$
for all $ u \in U_0 $ and $ h \in P $. Fix now $ h \in P $ and $ u \in U_0 $ and introduce the abbreviation $ \sigma = \sigma^0 + \sigma(\cdot) u $, then obviously
\begin{align*}
& \| S_t h + t \sigma^0(h) + t \sigma(h) u - S_t \operatorname{Fl}^{\sigma}_t (h) \|_{\beta} = t \bigg\|  \sigma(h) - S_t \frac{\operatorname{Fl}_t^{\sigma}(h) - h}{t} \bigg\|_{\beta},
\end{align*}
which means that
$$
\lim_{t \downarrow 0} \frac{1}{t} \| S_t h + t \sigma^0(h) + t \sigma(h) u  - S_t \operatorname{Fl}^{\sigma}_t (h) \|_{\beta} = 0.
$$
Hence Nagumo's condition can be equivalently formulated as
$$
\liminf_{t \downarrow 0} \frac{1}{t} d_P(S_t \operatorname{Fl}^{\sigma}_t(h)) = 0,
$$
for the particular choice of $ u $ and $ h \in P $. Due to conditions (\ref{cond-alpha-pre}) and (\ref{cond-sigma-pre}) the semiflow $ \operatorname{Fl}^{\sigma} $ leaves $ P $ invariant by \cite{vol:73}, the semigroup $ S_t $ certainly, too, therefore $ d_P(S_t \operatorname{Fl}^{\sigma}_t (h) ) = 0 $ and whence Nagumo's condition is more than satisfied.
\end{proof}

\begin{remark}
Having in mind the method of the moving frame, we could also have shown the Nagumo condition by argueing with time-dependent versions of Section 4 of \cite{Redheffer-Walter} or \cite{vol:73}. The method of the moving frame would work well in this particular case, since the convex set $ \cap_{t \in \mathbb{R}} S_t P $ of functions positive on the whole real line $ \mathbb{R} $ is invariant under the action of the shift $ S $.
\end{remark}

\begin{proposition}\label{prop-pos-general}
Suppose Assumptions \ref{ass-1-mild-hom}, \ref{ass-2-mild-hom},
\ref{ass-smooth} and conditions (\ref{cond-integrable-pre}),
(\ref{cond-alpha-pre}), (\ref{cond-sigma-pre}),
(\ref{cond-gamma-pre}) are fulfilled. Then, equation
(\ref{equation-new}) is positivity preserving.
\end{proposition}

\begin{proof}
Since the measure $F$ is $\sigma$-finite, there exists a sequence
$(B_n)_{n \in \mathbb{N}} \subset \mathcal{E}$ with $B_n \uparrow E$
and $F(B_n) < \infty$ for all $n \in \mathbb{N}$. Let $h_0 \in
L^2(\Omega,\mathcal{F}_0,\mathbb{P};H_{\beta})$ be arbitrary.
Relations (\ref{cond-alpha-pre}), (\ref{observation}),
(\ref{cond-sigma-pre}), Proposition \ref{prop-inv-Wiener} and
(\ref{cond-gamma-pre}) together with the closedness of $P$ yield
that, for each $n \in \mathbb{N}$, the mild solution $(r_t^n)_{t
\geq 0}$ of the stochastic partial differential equation
\begin{align}\label{equation-n}
\left\{
\begin{array}{rcl}
dr_t^n & = & ( \frac{d}{d \xi} r_t^n + \alpha(r_t^n) - \int_{B_n}
\gamma(r_t^n,x) ) dt + \sum_j \sigma^j(r_t^n)d \beta_t^j
\\ && + \int_{B_n} \gamma(r_{t-}^n,x)\mu(dt,dx)
\\ r_0^n & = & h_0
\end{array}
\right.
\end{align}
satisfies $\mathbb{P}(\bigcap_{t \in \mathbb{R}_+} r_{t \wedge
\tau}^n \in P) = 1$, where $\tau$ denotes the strictly positive
stopping time
\begin{align*}
\tau := \inf \{ t > 0 : \mu([0,t] \times B_n) = 1 \}.
\end{align*}
By virtue of Lemma \ref{lemma-local-global}, for each $n \in
\mathbb{N}$ equation (\ref{equation-n}) is positivity preserving.
According to \cite[Prop. 9.1]{SPDE} we have
\begin{align*}
\mathbb{E} \bigg[ \sup_{t \in [0,T]} \| r_t - r_t^n \|_{\beta}^2
\bigg] \rightarrow 0 \quad \text{for all $T \in \mathbb{R}_+$,}
\end{align*}
proving that equation (\ref{equation-new}) is positivity preserving.
\end{proof}

Finally the next theorem states the sufficient conditions under which we can conclude that the solution of equation \eqref{equation-new} is positivity preserving.

\begin{theorem}\label{thm-pos}
Suppose Assumptions \ref{ass-1-mild-hom}, \ref{ass-2-mild-hom},
\ref{ass-smooth} are fulfilled. Then, for each initial curve $h_0
\in L^2(\Omega,\mathcal{F}_0,\mathbb{P};H_{\beta})$ there exists a
unique adapted, c\`adl\`ag, mean square continuous
$\mathcal{H}_{\beta}$-valued solution $(f_t)_{t \geq 0}$ for the HJM
equation (\ref{HJM-eqn-new}) with $f_0 = \ell h_0$ satisfying
(\ref{solution-in-S2-f}), and there exists a unique adapted,
c\`adl\`ag, mean square continuous mild and weak $H_{\beta}$-valued
solution $(r_t)_{t \geq 0}$ for the HJMM equation
(\ref{equation-new}) with $r_0 = h_0$ satisfying
(\ref{solution-in-S2}), which is given by $r_t := \pi U_t f_t$, $t
\geq 0$. Moreover, equation (\ref{equation-new}) is positivity
preserving if and only if we have (\ref{cond-integrable-pre}),
(\ref{cond-alpha-pre}), (\ref{cond-sigma-pre}),
(\ref{cond-gamma-pre}).
\end{theorem}

\begin{proof}
The statement follows from Theorem \ref{thm-ex-SDE}, Proposition
\ref{prop-pos-general-nec} (see also Remark \ref{remark-continuity})
and Proposition \ref{prop-pos-general}.
\end{proof}

\begin{remark}\label{remark-pos}
Note that Theorem \ref{thm-pos} is also valid on other state spaces.
The only essential requirement is that the Hilbert space $H$
consists of continuous, real-valued functions on which the point
evaluations are continuous functionals, and that the shift semigroup
extends to a strongly continuous group in the sense of Assumption
\ref{ass-group}.
\end{remark}

\begin{remark}
For the particular situation where equation (\ref{equation-new}) has
no jumps, Theorem \ref{thm-pos} corresponds to the statement of
\cite[Thm. 3]{Milian-inf}, where positivity on weighted $L^2$-spaces
is investigated. Since point evaluations are discontinuous
functionals on $L^2$-spaces, the conditions in \cite{Milian-inf} are
formulated by taking other appropriate functionals.
\end{remark}

We shall now consider the arbitrage free situation. Let $\alpha =
\alpha_{\rm HJM} : H_{\beta} \rightarrow H_{\beta}$ in
(\ref{equation-new}) be defined according to the HJM drift condition
(\ref{def-alpha-HJM-intro}).

\begin{proposition}\label{prop-arbitrage-free}
Conditions (\ref{cond-integrable-pre}), (\ref{cond-alpha-pre}),
(\ref{cond-sigma-pre}), (\ref{cond-gamma-pre}) are satisfied if and
only if we have (\ref{cond-sigma-pre}), (\ref{cond-gamma-pre}) and
\begin{align}\label{cond-gamma-tsm}
\gamma(h,x)(\xi) = 0, \quad \xi \in (0,\infty), \, h \in
\partial P_{\xi} \text{ and $F$-almost all $x \in E$.}
\end{align}
\end{proposition}

\begin{proof}
Provided (\ref{cond-sigma-pre}), (\ref{cond-gamma-pre}) are
fulfilled, conditions (\ref{cond-integrable-pre}),
(\ref{cond-alpha-pre}) are satisfied if and only if we have
(\ref{cond-integrable-pre}) and
\begin{align}\label{observation-2}
- \int_E \gamma(h,x)(\xi) e^{\Gamma(h,x)(\xi)} F(dx) \geq 0, \quad
\xi \in (0,\infty), \, h \in \partial P_{\xi}
\end{align}
because the drift $\alpha$ is given by (\ref{def-alpha-HJM-intro}).
By (\ref{observation}), relations (\ref{cond-integrable-pre}),
(\ref{observation-2}) are fulfilled if and only if we have
(\ref{cond-gamma-tsm}).
\end{proof}

Now let, as in Section \ref{sec-ex-tsm}, coefficients $\sigma :
H_{\beta} \rightarrow L_2^0(H_{\beta}^0)$ and $\gamma : H_{\beta}
\times E \rightarrow H_{\beta'}^0$ be given, where $\beta' > \beta$
is a real number.

\begin{theorem}\label{thm-ex-pos}
Suppose Assumption \ref{ass-Lipschitz} is fulfilled, suppose furthermore that $\sigma \in C^2(H;L_2^0(H))$, and that the vector field
$$
h \mapsto - \frac{1}{2} \sum_{j} D \sigma^j(h) \sigma^j(h)
$$
is globally Lipschitz from $ H $ to $ H $. Then, the
statement of Theorem \ref{thm-ex-hjm} is valid, and, in addition,
the HJMM equation (\ref{HJMM-eqn}) is positivity preserving if and
only if we have (\ref{cond-sigma-pre}), (\ref{cond-gamma-pre}),
(\ref{cond-gamma-tsm}).
\end{theorem}

\begin{proof}
The statement follows from Theorem \ref{thm-ex-hjm}, Theorem
\ref{thm-pos} and Proposition \ref{prop-arbitrage-free}.
\end{proof}

Finally, let us consider the L\'evy case, treated at the end of
Section \ref{sec-ex-tsm}. In this framework, the following statement
is valid.

\begin{proposition}\label{prop-pos-Levy}
Conditions (\ref{cond-sigma-pre}), (\ref{cond-gamma-pre}) and
(\ref{cond-gamma-tsm}) are satisfied if and only if we have
(\ref{cond-sigma-pre}) and
\begin{align}\label{cond-Levy-1}
& h + \delta_k(h) x \in P, \quad h \in P, \, k = 1,\ldots,e \text{
and $F_k$-almost all $x \in \mathbb{R}$}
\\ \label{cond-Levy-2} &\delta_k(h)(\xi) = 0, \quad \xi \in (0,\infty), \, h \in
\partial P_{\xi} \text{ and all $k = 1,\ldots,e$ with $F_k(\mathbb{R}) > 0$.}
\end{align}
\end{proposition}

\begin{proof}
The claim follows from the definition (\ref{def-Levy-measure}) of
$F$ and the definition (\ref{def-gamma-Levy}) of $\gamma$.
\end{proof}

\begin{corollary}\label{cor-ex-pos}
Suppose Assumption \ref{ass-cor} is fulfilled, suppose furthermore that $\sigma \in C^2(H;L_2^0(H))$, and that the vector field
$$
h \mapsto  - \frac{1}{2} \sum_{j} D \sigma^j(h) \sigma^j(h),
$$
is globally Lipschitz from $ H $ to $ H $. Then, the statement
of Corollary \ref{cor-ex-HJM} is valid, and, in addition, the HJMM
equation (\ref{HJMM-eqn}) is positivity preserving if and only if we
have (\ref{cond-sigma-pre}), (\ref{cond-Levy-1}),
(\ref{cond-Levy-2}).
\end{corollary}

\begin{proof}
The assertion follows from Theorem \ref{thm-ex-pos} and Proposition
\ref{prop-pos-Levy}.
\end{proof}

Our above results on arbitrage free, positivity preserving term
structure models apply in particular for local state dependent
volatilities. The following two results are obvious.

\begin{proposition}
Suppose for all $j$ there are $\tilde{\sigma}^j : \mathbb{R}_+
\times \mathbb{R} \rightarrow \mathbb{R}$, and $\tilde{\gamma} :
\mathbb{R}_+ \times \mathbb{R} \times E \rightarrow \mathbb{R}$ such
that
\begin{align}\label{sigma-state}
\sigma^j(h)(\xi) &= \tilde{\sigma}^j(\xi,h(\xi)), \quad (h,\xi) \in
H_{\beta} \times \mathbb{R}_+, \quad \text{for all $j$}
\\ \gamma(h,x)(\xi) &= \tilde{\gamma}(\xi,h(\xi),x), \quad (h,x,\xi) \in
H_{\beta} \times E \times \mathbb{R}_+.
\end{align}
Then, conditions (\ref{cond-sigma-pre}), (\ref{cond-gamma-pre}),
(\ref{cond-gamma-tsm}) are fulfilled if and only if
\begin{align}\label{sigma-local}
&\tilde{\sigma}^j(\xi,0) = 0, \quad \xi \in (0,\infty), \quad
\text{for all $j$}
\\ &y + \tilde{\gamma}(\xi,y,x) \geq 0, \quad \xi \in (0,\infty), \, y \in \mathbb{R}_+
\text{ and $F$-almost all $x \in E$}
\\ &\tilde{\gamma}(\xi,0,x) = 0, \quad \xi \in (0,\infty) \text{ and $F$-almost all $x \in E$.}
\end{align}
\end{proposition}

L\'evy term structure models with local state dependent volatilities
have been studied in \cite{P-Z-paper} and \cite{Marinelli}. In the
framework of Proposition \ref{prop-pos-Levy} we obtain:

\begin{proposition}\label{prop-pos-Levy-state}
Suppose for all $j$ there are $\tilde{\sigma}^j : \mathbb{R}_+
\times \mathbb{R} \rightarrow \mathbb{R}$, and for all $k =
1,\ldots,e$ there are $\tilde{\delta}_k : \mathbb{R}_+ \times
\mathbb{R} \rightarrow \mathbb{R}$ such that we have
(\ref{sigma-state}) and
\begin{align}
\delta_k(h)(\xi) = \tilde{\delta}_k(\xi,h(\xi)), \quad (h,\xi) \in
H_{\beta} \times \mathbb{R}_+, \quad k = 1,\ldots,e.
\end{align}
Then, conditions (\ref{cond-sigma-pre}), (\ref{cond-Levy-1}),
(\ref{cond-Levy-2}) are fulfilled if and only if we have
(\ref{sigma-local}) and
\begin{align}
& y + \tilde{\delta}_k(\xi,y) x \geq 0, \quad \xi \in (0,\infty), \,
y \in \mathbb{R}_+, \, k = 1,\ldots,e \text{ and $F_k$-almost all $x
\in \mathbb{R}$}
\\ &\tilde{\delta}_k(\xi,0) = 0, \quad \xi
\in (0,\infty) \text{ and all $k = 1,\ldots,e$ with $F_k(\mathbb{R})
> 0$.}
\end{align}
\end{proposition}

Section 5 in \cite{P-Z-paper} contains some positivity results for
L\'evy driven term structure models on weighted $L^2$-spaces. Using
Proposition \ref{prop-pos-Levy-state}, we can derive the analogous
statements of \cite[Thm. 4]{P-Z-paper} on our $H_{\beta}$-spaces.

\section{The Brody-Hughston equation: Existence and uniqueness}\label{sec_bh-equation}

Let $ (H,\| \cdot \|) $ now denote a state (Hilbert) space of
continuous, integrable real-valued maps on $ \mathbb{R} $, where the
shift semigroup acts as a strongly continuous group, for instance $
H^1(\mathbb{R}) $. We need one notation for the sake of simplicity:
Let $ \rho $ denote a probability density on $ \mathbb{R}_+ $,
extended by $ 0 $ to the whole real line, and $ \eta \in
L^1(\mathbb{R},\rho(\cdot)du) $.

\begin{assumption}\label{ass_bh}
We assume $\int_E \| \gamma(0,x) \|^2 F(dx) < \infty$ and that there exists a constant $L > 0$ such that
\begin{align}
\| \sigma(h_1) - \sigma(h_2) \|_{L_2^0(H)} &\leq L \| h_1 - h_2 \|,
\\ \bigg( \int_E \| \gamma(h_1,x)
- \gamma(h_2,x) \|^2 F(dx) \bigg)^{\frac{1}{2}} &\leq L \| h_1 - h_2
\|
\end{align}
for all $h_1, h_2 \in H$.

Furthermore we assume for all $ \rho \in H $ that
\begin{align}
\int_0^\infty \sigma(\rho)(u) d u = 0, 
\end{align}
and we assume that
\begin{align}
\sigma(\rho)(\xi) = 0
\end{align}
for all $ \rho \in P $, $ \xi \geq 0 $ and $ \rho(\xi) = 0 $.

For the jump fields we assume that
\begin{align}
& \rho + \gamma(\rho,x) \in P, \, \text{for all $\rho \in P$ and
$F$-almost all $x \in E$,} \\
&  -\int_E \gamma(\rho,x) F(dx)(\xi) \geq 0, \, \text{for all $\xi \in (0,\infty)$, $\rho \in \partial P_{\xi}$,}
\end{align}
and finally that
\begin{align}
\int_0^{\infty} \gamma(\rho,x)(u) du = 0
\end{align}
for all $\rho \in H$ and $F$-almost all $x \in E$.
\end{assumption}

Under theses assumptions we can prove the following theorem:
\begin{theorem}
The following equation, which we call henceforward Brody-Hughston equation,
\begin{align}\label{BH-equation}
\left\{
\begin{array}{rcl}
d \rho_t & = & ( \frac{d}{d\xi} \rho_t + \rho_t (0) \rho_t ) dt +
\sum_{j} \sigma^j(\rho_t) d\beta^j_t
\\ && + \gamma(\rho_{t-},x) (\mu(dt,dx) - F(dx) dt) \medskip
\\ \rho_0 & \in & H,
\end{array}
\right.
\end{align}
has a unique adapted, c\`adl\`ag, mean square continuous mild and
weak solution for all times in $ H $, which leaves the set of
densities invariant. Furthermore
\begin{align*}
P(t,T) = \int_{T-t}^{\infty} \rho(t,u) du
\end{align*}
for $ 0 \leq t \leq T $ defines an arbitrage-free evolution of bond prices, i.e. the discounted bond price processes
\begin{align*}
\exp \bigg( - \int_0^t \rho_s(0) ds \bigg) \int_{T-t}^{\infty}
\rho(t,u) du
\end{align*}
are local martingales for $ 0 \leq t \leq T $.
\end{theorem}

\begin{proof}
The proof is a direct application of Theorem \ref{thm-pos} (see also
Remark \ref{remark-pos}). The mass of $ \rho_t $ is preserved since applying the linear functional
$$
\rho \mapsto \int_0^{\infty} \rho(u) du
$$
makes the right hand side vanish.
\end{proof}

\begin{remark}
Notice the conceptual simplicity of the Brody-Hughston equation
(\ref{BH-equation}) in contrast to the HJM equation. In particular,
adding jumps to the Brody-Hughston equation is less delicate than in the HJMM case. Remark also that positivity is the crucial issue for the HJMM equation as well
as for the Brody-Hughston equation.
\end{remark}

\begin{remark}
We can define the average of $ \eta $ at $ \rho $
$$
\overline{\eta} = \int_0^{\infty} \eta (u) \rho(u) du.
$$
We suppress in this notation the dependence on $ \rho $, but it
should be clear at every moment, where it appears, which $ \rho $ is
meant. Vector fields, volatilities or jump fields, satisfying Assumptions \ref{ass_bh} can then be chosen of the form
$$
\rho \mapsto (a(\rho) - \overline{(a(\rho)})\rho,
$$
for some $ a:H \to H $ appropriately chosen. In this case, i.e., $ \sigma^j $ and $ \gamma $ chosen of the previous type, one can try to divide the equation by $ \rho $ in order to show positivity directly, which works up to some regularity questions. Our approach chosen here is more general since we do not need to assume that the vector fields factor by $ \rho $.
\end{remark}

\begin{appendix}

\section{Stochastic partial differential equations driven by Wiener process and
Poisson measures}\label{app-SDE}

For convenience of the reader, we provide the crucial results on
stochastic partial differential equations driven by Wiener process
and Poisson measures in this appendix. For this purpose, we follow
\cite{SPDE}, where we understand stochastic partial differential
equations -- in this paper, the HJMM equation (\ref{HJMM-eqn}) -- as
time-dependent transformations of stochastic differential equations
-- in this text, the HJM equation (\ref{HJM-eqn}). Other references
for existence and uniqueness results on stochastic partial
differential equations driven by Wiener process and Poisson measures
are \cite{Ruediger-mild} and \cite{Marinelli-Prevot-Roeckner}.

Let $H$ denote a separable Hilbert space with inner product $\langle
\cdot,\cdot \rangle$ and associated norm $\| \cdot \|$.

Furthermore, let $(S_t)_{t \geq 0}$ be a $C_0$-semigroup on $H$ with
infinitesimal generator $A : \mathcal{D}(A) \subset H \rightarrow
H$. We denote by $A^* : \mathcal{D}(A^*) \subset H \rightarrow H$
the adjoint operator of $A$. Recall that the domains
$\mathcal{D}(A)$ and $\mathcal{D}(A^*)$ are dense in $H$, see, e.g.,
\cite[Satz VII.4.6, p. 351]{Werner}.

Let $(\Omega,\mathcal{F},(\mathcal{F}_t)_{t \geq 0},\mathbb{P})$ be
a filtered probability space satisfying the usual conditions.

Let $U$ be another separable Hilbert space. Whenever there is no ambiguity possible, we also denote by $\langle \cdot,\cdot \rangle$ its inner product, and by $\| \cdot \|$ its associated norm. Let $Q \in L(U)$ be a compact,
self-adjoint, strictly positive linear operator. Then there exist an
orthonormal basis $\{ e_j \}$ of $U$ and a bounded sequence
$\lambda_j$ of strictly positive real numbers such that
\begin{align*}
Qu = \sum_j \lambda_j \langle u,e_j \rangle e_j, \quad u \in U
\end{align*}
namely, the $\lambda_j$ are the eigenvalues of $Q$, and each $e_j$
is an eigenvector corresponding to $\lambda_j$, see, e.g.,
\cite[Thm. VI.3.2]{Werner}.

The space $U_0 := Q^{\frac{1}{2}}(U)$, equipped with inner product
$\langle u,v \rangle_{U_0} := \langle Q^{-\frac{1}{2}} u,
Q^{-\frac{1}{2}} v \rangle_U$, is another separable Hilbert space
and $\{ \sqrt{\lambda_j} e_j \}$ is an orthonormal basis.

Let $W$ be a $Q$-Wiener process \cite[p. 86,87]{Da_Prato}. We assume
that ${\rm Tr} \, Q = \sum_j \lambda_j < \infty$. Otherwise, which
is the case if $W$ is a cylindrical Wiener process, there always
exists a separable Hilbert space $U_1 \supset U$ on which $W$ has a
realization as a finite trace class Wiener process, see \cite[Chap.
4.3]{Da_Prato}.

We denote by $L_2^0(H) := L_2(U_0,H)$ the space of Hilbert-Schmidt
operators from $U_0$ into $H$, which, endowed with the
Hilbert-Schmidt norm
\begin{align*}
\| \Phi \|_{L_2^0(H)} := \sqrt{\sum_j \lambda_j \| \Phi e_j \|^2},
\quad \Phi \in L_2^0(H)
\end{align*}
itself is a separable Hilbert space.

According to \cite[Prop. 4.1]{Da_Prato}, the sequence of stochastic
processes $\{ \beta^j \}$ defined as $\beta^j :=
\frac{1}{\sqrt{\lambda_j}} \langle W, e_j \rangle$ is a sequence of
real-valued independent $(\mathcal{F}_t)$-Brownian motions and we
have the expansion
\begin{align}\label{Wiener-expansion}
W = \sum_j \sqrt{\lambda _j} \beta^j e_j,
\end{align}
where the series is convergent in the space $M^2(U)$ of $U$-valued
square-integrable martingales. Let $\Phi : \Omega \times
\mathbb{R}_+ \rightarrow L_2^0(H)$ be an integrable process, i.e.
$\Phi$ is predictable and satisfies
\begin{align*}
\mathbb{P} \bigg( \int_0^T \| \Phi_t \|_{L_2^0(H)}^2 dt < \infty
\bigg) = 1 \quad \text{for all $T \in \mathbb{R}_+$.}
\end{align*}
Setting $\Phi^j := \sqrt{\lambda_j} \Phi e_j$ for each $j$, we have
\begin{align}\label{Wiener-int-expansion}
\int_0^t \Phi_s dW_s = \sum_j \int_0^t \Phi_s^j d\beta_s^j, \quad t
\in \mathbb{R}_+
\end{align}
where the convergence is uniformly on compact time intervals in
probability, see \cite[Thm. 4.3]{Da_Prato}.

Let $(E,\mathcal{E})$ be a measurable space which we assume to be a
\textit{Blackwell space} (see \cite{Dellacherie,Getoor}). We remark
that every Polish space with its Borel $\sigma$-field is a Blackwell
space.

Furthermore, let $\mu$ be a homogeneous Poisson random measure on
$\mathbb{R}_+ \times E$, see \cite[Def. II.1.20]{Jacod-Shiryaev}.
Then its compensator is of the form $dt \otimes F(dx)$, where $F$ is
a $\sigma$-finite measure on $(E,\mathcal{E})$.

We shall now focus on (semi-linear) stochastic partial differential
equations
\begin{align}\label{equation}
\left\{
\begin{array}{rcl}
dr_t & = & (A r_t + \alpha(r_t))dt + \sum_{j}
\sigma^j(r_t)d\beta_t^j + \int_{E} \gamma(r_{t-},x) (\mu(dt,dx) -
F(dx)dt)
\medskip
\\ r_0 & = & h_0
\end{array}
\right.
\end{align}
on the separable Hilbert space $H$ with coefficients $\alpha : H
\rightarrow H$, $\sigma : H \rightarrow L_2^0(H)$ and $\gamma : H
\times E \rightarrow H$. In (\ref{equation}), we have defined
$\sigma^j : H \rightarrow H$ as $\sigma^j(h) := \sqrt{\lambda_j}
\sigma(h) e_j$ for each $j$. The initial condition is an
$\mathcal{F}_0$-measurable random variable $h_0 : \Omega \rightarrow
H$.

\begin{definition}\label{def-strong}
An adapted, c\`adl\`ag $H$-valued process $(r_t)_{t \geq 0}$ is
called a {\rm strong solution} for (\ref{equation}) with $r_0 = h_0$
if we have $r_t \in \mathcal{D}(A)$, $t \geq 0$, the relation
\begin{align*}
\mathbb{P} \bigg( \int_0^t \bigg( \| A r_s + \alpha(r_s) \| + \|
\sigma(r_s) \|_{L_2^0(H)}^2 + \int_E \| \gamma(r_s,x) \|^2 F(dx)
\bigg) ds < \infty \bigg) = 1
\end{align*}
for all $t \in \mathbb{R}_+$, and
\begin{align*}
r_t &= h_0 + \int_0^t (A r_s + \alpha(r_s))ds + \sum_{j} \int_0^t
\sigma^j(r_s) d\beta_s^j
\\ &\quad + \int_0^t \int_{E} \gamma(r_{s-},x) (\mu(ds,dx) - F(dx)ds), \quad t \geq 0.
\end{align*}
\end{definition}

\begin{definition}\label{def-weak}
An adapted, c\`adl\`ag $H$-valued process $(r_t)_{t \geq 0}$ is
called a {\rm weak solution} for (\ref{equation}) with $r_0 = h_0$
if
\begin{align}\label{int-cond-weak}
\mathbb{P} \bigg( \int_0^t \bigg( \| \alpha(r_s) \| + \| \sigma(r_s)
\|_{L_2^0(H)}^2 + \int_E \| \gamma(r_s,x) \|^2 F(dx) \bigg) ds <
\infty \bigg) = 1
\end{align}
for all $t \in \mathbb{R}_+$, and for all $\zeta \in
\mathcal{D}(A^*)$ we have
\begin{align*}
\langle \zeta,r_t \rangle &= \langle \zeta,h_0 \rangle + \int_0^t (
\langle A^* \zeta, r_s \rangle + \langle \zeta, \alpha(r_s)
\rangle)ds + \sum_{j} \int_0^t \langle \zeta, \sigma^j(r_s) \rangle
d\beta_s^j
\\ &\quad + \int_0^t \int_{E} \langle \zeta, \gamma(r_{s-},x) \rangle
(\mu(ds,dx) - F(dx)ds), \quad t \geq 0.
\end{align*}
\end{definition}

\begin{definition}\label{def-mild}
An adapted, c\`adl\`ag $H$-valued process $(r_t)_{t \geq 0}$ is
called a {\rm mild solution} for (\ref{equation}) with $r_0 = h_0$
if we have (\ref{int-cond-weak}) for all $t \in \mathbb{R}_+$, and
\begin{align*}
r_t &= S_t h_0 + \int_0^t S_{t-s} \alpha(r_s) ds + \sum_{j} \int_0^t
S_{t-s} \sigma^j(r_s) d\beta_s^j
\\ &\quad + \int_0^t \int_{E} S_{t-s} \gamma(r_{s-},x) (\mu(ds,dx) - F(dx)ds), \quad t \geq 0.
\end{align*}
\end{definition}

By convention, \textit{uniqueness} of solutions for (\ref{equation})
is meant up to indistinguishability, that is, for two solutions
$r^1, r^2$ we have $\mathbb{P}(\bigcap_{t \in \mathbb{R}_+} \{ r_t^1
= r_t^2 \}) = 1$.

\begin{assumption}\label{ass-group}
There exist another separable Hilbert space $\mathcal{H}$, a
$C_0$-group $(U_t)_{t \in \mathbb{R}}$ on $\mathcal{H}$ and
continuous linear operators $\ell \in L(H,\mathcal{H})$, $\pi \in
L(\mathcal{H},H)$ such that the diagram
\[ \begin{CD}
\mathcal{H} @>U_t>> \mathcal{H}\\
@AA\ell A @VV\pi V\\
H @>S_t>> H
\end{CD} \]
commutes for every $t \in \mathbb{R}_+$, that is
\begin{align*}
\pi U_t \ell h = S_t h \quad \text{for all $t \in \mathbb{R}_+$ and
$h \in H$.}
\end{align*}
\end{assumption}

\begin{assumption}\label{ass-1-mild-hom}
We assume $\int_E \| \gamma(0,x) \|^2 F(dx) < \infty$.
\end{assumption}

\begin{assumption}\label{ass-2-mild-hom}
We assume there is a constant $L > 0$ such that
\begin{align*}
\| \alpha(h_1) - \alpha(h_2) \| &\leq L \| h_1 - h_2 \|,
\\ \| \sigma(h_1) - \sigma(h_2) \|_{L_2^0(H)} &\leq
L \| h_1 - h_2 \|,
\\ \bigg( \int_E \| \gamma(h_1,x)
- \gamma(h_2,x) \|^2 F(dx) \bigg)^{\frac{1}{2}} &\leq L \| h_1 - h_2
\|
\end{align*}
for all $h_1, h_2 \in H$.
\end{assumption}

\begin{theorem}\label{thm-ex-SDE}
\cite[Thm. 8.6, Cor. 10.6]{SPDE} Suppose that Assumptions
\ref{ass-group}, \ref{ass-1-mild-hom}, \ref{ass-2-mild-hom} are
fulfilled. Then, for each $h_0 \in
L^2(\Omega,\mathcal{F}_0,\mathbb{P};H)$ there exists a unique
c\`adl\`ag, adapted, mean square continuous mild and weak
$\mathcal{H}$-valued solution $(f_t)_{t \geq 0}$ for
\begin{equation}\label{equation-f}
\begin{aligned}
df_t &= U_{-t} \ell \alpha(\pi U_t f_t)dt + \sum_j U_{-t} \ell
\sigma^j(\pi U_t f_t) d\beta_t^j
\\ &\quad + \int_E U_{-t}
\ell \gamma(\pi U_t f_{t-},x) (\mu(dt,dx) - F(dx)dt)
\end{aligned}
\end{equation}
with $f_0 = \ell h_0$ satisfying
\begin{align*}
\mathbb{E} \bigg[ \sup_{t \in [0,T]} \| f_t \|^2 \bigg] < \infty
\quad \text{for all $T \in \mathbb{R}_+$,}
\end{align*}
and there exists a unique c\`adl\`ag, adapted, mean square
continuous mild and weak $H$-valued solution $(r_t)_{t \geq 0}$ for
(\ref{equation}) with $r_0 = h_0$ satisfying
\begin{align*}
\mathbb{E} \bigg[ \sup_{t \in [0,T]} \| r_t \|^2 \bigg] < \infty
\quad \text{for all $T \in \mathbb{R}_+$,}
\end{align*}
which is given by $r_t := \pi U_t f_t$, $t \geq 0$.
\end{theorem}

\begin{remark}\label{remark-solution}
By a solution $(f_t)_{t \geq 0}$ for (\ref{equation-f}) with $f_0 =
\ell h_0$ we precisely mean that
\begin{align*}
&\mathbb{P} \bigg( \int_0^t \bigg( \| U_{-s} \ell \alpha(\pi U_s
f_s) \| + \| U_{-s} \ell \sigma(\pi U_s f_s) \|_{L_2^0(H)}^2
\\ &\quad \quad \quad + \int_E
\| U_{-s} \ell \gamma(\pi U_s f_s,x) \|^2 F(dx) \bigg) ds < \infty
\bigg) = 1
\end{align*}
for all $t \in \mathbb{R}_+$, and we have
\begin{align*}
f_t &= \ell h_0 + \int_0^t U_{-s} \ell \alpha(\pi U_s f_s)ds +
\sum_j \int_0^t U_{-s} \ell \sigma^j(\pi U_s f_s) d\beta_s^j
\\ &\quad + \int_0^t \int_E U_{-s}
\ell \gamma(\pi U_s f_{s-},x) (\mu(ds,dx) - F(dx)ds), \quad t \geq
0.
\end{align*}
\end{remark}

\begin{lemma}\label{lemma-solution-tau}
Let $\tau$ be a bounded stopping time. We define the new filtration
$(\tilde{\mathcal{F}}_t)_{t \geq 0}$ by $\tilde{\mathcal{F}}_t :=
\mathcal{F}_{\tau + t}$, the new $U$-valued process $\tilde{W}$ by
$\tilde{W}_t := W_{\tau + t} - W_{\tau}$ and the new random measure
$\tilde{\mu}$ on $\mathbb{R}_+ \times E$ by $\tilde{\mu}(\omega;B)
:= \mu(\omega;B_{\tau(\omega)})$, $B \in \mathcal{B}(\mathbb{R}_+)
\otimes \mathcal{E}$, where
\begin{align*}
B_{\tau} := \{ (t + \tau,x) \in \mathbb{R}_+ \times E : (t,x) \in B
\}.
\end{align*}
Then $\tilde{W}$ is a $Q$-Wiener process with respect to
$(\tilde{\mathcal{F}}_t)_{t \geq 0}$ and $\tilde{\mu}$ is a
homogeneous Poisson random measure on $\mathbb{R}_+ \times E$ with
respect to $(\tilde{\mathcal{F}}_t)_{t \geq 0}$ having the
compensator $dt \otimes F(dx)$. Moreover, we have the expansion
\begin{align}\label{Wiener-expansion-tau}
\tilde{W} = \sum_j \sqrt{\lambda _j} \tilde{\beta}^j e_j,
\end{align}
where $\tilde{\beta}^j$ defined as $\tilde{\beta}_t^j := \beta_{\tau
+ t}^j - \beta_{\tau}^j$ is a sequence of real-valued independent
$(\tilde{\mathcal{F}}_t)$-Brownian motions. Furthermore, if
$(r_t)_{t \geq 0}$ is a weak solution for (\ref{equation}), then the
$(\tilde{\mathcal{F}}_t)$-adapted process $(\tilde{r}_t)_{t \geq 0}$
defined by $\tilde{r}_t := r_{\tau + t}$ is a weak solution for
\begin{align}\label{SPDE-tau}
\left\{
\begin{array}{rcl}
d\tilde{r}_t & = & (A\tilde{r}_t + \alpha(\tilde{r}_t))dt + \sum_j
\sigma^j(\tilde{r}_t) d\tilde{\beta}_t^j + \int_E
\gamma(\tilde{r}_{t-},x) (\tilde{\mu}(dt,dx) - F(dx) dt)
\medskip
\\ \tilde{r}_0 & = & r_{\tau}.
\end{array}
\right.
\end{align}
\end{lemma}

\begin{proof}
Note that $\tilde{W}$ is a continuous
$(\tilde{\mathcal{F}}_t)$-adapted process with $\tilde{W}_0 = 0$,
and $\tilde{\mu}$ is an integer-valued random measure on
$\mathbb{R}_+ \times E$.

We fix $u \in U$. The process
\begin{align*}
M_t := \frac{\exp( i \langle u,W_t \rangle )}{\mathbb{E}[\exp(i
\langle u,W_t \rangle)]}, \quad t \geq 0
\end{align*}
is a complex-valued martingale, because for all $s,t \in
\mathbb{R}_+$ with $s < t$ the random variable $W_t - W_s$ and the
$\sigma$-algebra $\mathcal{F}_s$ are independent. The martingale
$(M_t)_{t \geq 0}$ admits the representation
\begin{align*}
M_t = \exp \bigg( i \langle u,W_t \rangle + \frac{t}{2} \langle Qu,u
\rangle \bigg), \quad t \geq 0.
\end{align*}
According to the Optional Stopping Theorem, the process $(M_{t +
\tau})_{t \geq 0}$ is a nowhere vanishing complex
$(\tilde{\mathcal{F}}_t)$-martingale. Thus, for $s,t \in
\mathbb{R}_+$ with $s<t$ we obtain
\begin{align*}
\mathbb{E} \bigg[ \frac{M_{t + \tau}}{M_{s + \tau}} \, | \,
\tilde{\mathcal{F}}_s \bigg] = 1.
\end{align*}
For each $C \in \tilde{\mathcal{F}}_s$ we get
\begin{align*}
\mathbb{E} [ \mathbbm{1}_C \exp ( i \langle u, \tilde{W}_t -
\tilde{W}_s \rangle ) ] = \mathbb{P}(C) \exp \bigg( - \frac{t-s}{2}
\langle Qu,u \rangle \bigg).
\end{align*}
Hence, the random variable $\tilde{W}_t - \tilde{W}_s$ and the
$\sigma$-algebra $\tilde{\mathcal{F}}_s$ are independent, and
$\tilde{W}_t - \tilde{W}_s$ has a Gaussian distribution with
covariance operator $(t-s) Q$. The expansion
(\ref{Wiener-expansion-tau}) follows from (\ref{Wiener-expansion}).

Now fix $v \in \mathbb{R}$ and $B \in \mathcal{E}$ with $F(B) <
\infty$. The process
\begin{align*}
N_t := \frac{\exp(i v \mu([0,t] \times B))}{\mathbb{E}[\exp(i v
\mu([0,t] \times B))]}, \quad t \geq 0
\end{align*}
is a complex-valued martingale, because for all $s,t \in
\mathbb{R}_+$ with $s < t$ the random variable $\mu((s,t] \times B)$
and the $\sigma$-algebra $\mathcal{F}_s$ are independent. By
\cite[Thm. II.4.8]{Jacod-Shiryaev} the martingale $(N_t)_{t \geq 0}$
admits the representation
\begin{align*}
N_t = \exp \Big( i v \mu([0,t] \times B) - (e^{i v} - 1) F(B) t
\Big), \quad t \geq 0.
\end{align*}
According to the Optional Stopping Theorem, the process $(N_{t +
\tau})_{t \geq 0}$ is a nowhere vanishing complex
$(\tilde{\mathcal{F}}_t)$-martingale. Thus, for $s,t \in
\mathbb{R}_+$ with $s<t$ we obtain
\begin{align*}
\mathbb{E} \bigg[ \frac{N_{t + \tau}}{N_{s + \tau}} \, | \,
\tilde{\mathcal{F}}_s \bigg] = 1.
\end{align*}
For each $C \in \tilde{\mathcal{F}}_s$ we get
\begin{align*}
\mathbb{E} [ \mathbbm{1}_C \exp ( i v \tilde{\mu}((s,t] \times B)) ]
= \mathbb{P}(C) \exp \Big( (e^{i v} - 1) F(B) (t-s) \Big).
\end{align*}
Hence, the random variable $\tilde{\mu}((s,t] \times B)$ and the
$\sigma$-algebra $\tilde{\mathcal{F}}_s$ are independent, and
$\tilde{\mu}((s,t] \times B)$ has a Poisson distribution with mean
$(t-s) F(B)$.

Next, we claim that
\begin{align}\label{int-W-tau}
\int_{\tau}^{\tau + t} \Phi_s dW_s = \int_0^t \Phi_{\tau + s} d
\tilde{W}_s
\end{align}
for every predictable process $\Phi : \Omega \times \mathbb{R}_+
\rightarrow L_2^0(H)$ satisfying
\begin{align*}
\mathbb{P} \bigg( \int_0^t \| \Phi_s \|_{L_2^0(H)}^2 ds < \infty
\bigg) = 1
\end{align*}
for all $t \in \mathbb{R}_+$, and
\begin{align}\label{int-mu-tau}
\int_{\tau}^{\tau + t} \Psi(s,x) \mu(ds,dx) = \int_0^t \Psi(\tau +
s,x) \tilde{\mu}(ds,dx)
\end{align}
for every predictable process $\Psi : \Omega \times \mathbb{R}_+
\times E \rightarrow L_2^0(H)$ satisfying
\begin{align*}
\mathbb{P} \bigg( \int_0^t \int_E \| \Psi(s,x) \|^2 F(dx) ds <
\infty \bigg) = 1
\end{align*}
for all $t \in \mathbb{R}_+$. If $\Phi$, $\Psi$ are elementary and
$\tau$ a simple stopping time, then (\ref{int-W-tau}),
(\ref{int-mu-tau}) hold by inspection. The general case follows by
localization.

If $(r_t)_{t \geq 0}$ is a weak solution to (\ref{equation}), for
every $\zeta \in \mathcal{D}(A^*)$ relations (\ref{int-W-tau}),
(\ref{int-mu-tau}) yield
\begin{align*}
\langle \zeta, r_{\tau + t} \rangle &= \langle \zeta, r_{\tau}
\rangle + \int_{\tau}^{\tau + t} ( \langle A^* \zeta, r_s \rangle +
\langle \zeta, \alpha(r_s) \rangle ) ds + \sum_{j} \int_{\tau}^{\tau
+ t} \langle \zeta, \sigma(r_s) \rangle d\beta_s^j
\\ &\quad + \int_{\tau}^{\tau + t} \int_{E} \langle
\zeta, \gamma(r_{s-},x) \rangle (\mu(ds,dx) - F(dx)ds)
\\ &= \langle \zeta, r_{\tau}
\rangle + \int_{0}^{t} ( \langle A^* \zeta, r_{\tau + s} \rangle +
\langle \zeta, \alpha(r_{\tau + s}) \rangle ) ds + \sum_{j}
\int_{0}^{t} \langle \zeta, \sigma(r_{\tau + s}) \rangle
d\tilde{\beta}_s^j
\\ &\quad + \int_{0}^{t} \int_{E} \langle
\zeta, \gamma(r_{(\tau + s)-},x) \rangle (\tilde{\mu}(ds,dx) -
F(dx)ds).
\end{align*}
Hence, $(\tilde{r}_t)_{t \geq 0}$ is a weak solution for
(\ref{SPDE-tau}).
\end{proof}

\end{appendix}

\end{document}